\documentclass{amsart} 

\usepackage{pstricks}
\include{pst-plot}
\usepackage{amsmath,amsthm}     
\usepackage{amssymb}            
\usepackage{euscript}           
\usepackage{enumerate,calc,graphicx,lscape,color}
\definecolor{red}{rgb}{1,0,0}
\usepackage[matrix,arrow,curve,frame]{xy}    

\xymatrixcolsep{1.9pc}                          
\xymatrixrowsep{1.9pc}
\newdir{ >}{{}*!/-5pt/\dir{>}}                  

\newtheorem{thm}[subsection]{Theorem}
\newtheorem{defn}[subsection]{Definition}
\newtheorem{prop}[subsection]{Proposition}

\newtheorem{lemma}[subsection]{Lemma}

\theoremstyle{definition}  
\newtheorem{ex}[subsection]{Example}

\newtheorem{remark}[subsection]{Remark}

  {\end{list}}

\newcommand{\dfn}{\textbf} 

\newcommand{\mdfn}[1]{\dfn{\mathversion{bold}#1}} 

\newcommand{\Smash}             {\wedge}

\newcommand{\tens}              {\otimes}               
\newcommand{\ctens}             {\,\hat{\otimes}\,}               
\newcommand{\iso}               {\cong}

\newcommand{\cat}{\EuScript}    
\newcommand{\cA}{{\cat A}}      
\newcommand{\cC}{{\cat C}}
\newcommand{\cD}{{\cat D}}
\newcommand{\cE}{{\cat E}}

\newcommand{\cU}{{\cat U}}

\newcommand{\Spe}{{\cat Sp}}
\newcommand{\MotSp}{{\cat MotSp}}

\DeclareMathOperator{\tr}{tr}

\newcommand{\Ho}{\text{Ho}\,}

\newcommand{\field}[1]  {\mathbb #1} 
\newcommand{\A}         {\field A}

\newcommand{\N}         {\field N}

\newcommand{\Z}         {\field Z}
\newcommand{\C}         {\field C}

\DeclareMathOperator{\Hom}{Hom}

\DeclareMathOperator{\End}{End}

\DeclareMathOperator{\ob}{ob}

\DeclareMathOperator{\id}{id}
\DeclareMathOperator{\Aut}{Aut}

\newcommand{\ra}{\rightarrow}                   
\newcommand{\lra}{\longrightarrow}              
\newcommand{\la}{\leftarrow}                    
\newcommand{\llra}[1]{\stackrel{#1}{\lra}}      

\newcommand{\blank}{-}                          
\newcommand{\und}{\underline}

\newcommand{\rea}[1]{|{#1}|}             

\newcommand{\ceck}[1]{\Cech(#1)}         
\newcommand{\oceck}[1]{\Cech^{o}(#1)}    
\newcommand{\oreal}[1]{\rea{\oceck{U}}}  
\newcommand{\creal}[1]{\rea{\ceck{U}}}   

\newcommand{\Cech}{\check{C}}

\newcommand{\power}[2]{\und{#1}^{#2}}

\newcommand{\psub}[2]{\phantom{\!}_{#1}{#2}}

\numberwithin{equation}{subsection}

\newenvironment{myequation}
  {\addtocounter{subsection}{1}\begin{eqnarray}}
  {\end{eqnarray}$\!\!$}

\DeclareMathOperator{\Pic}{Pic}

\DeclareMathOperator{\KL}{KL}

\newcommand{\um}[2]{\und{#1}^{{#2}}}
\newcommand{\GrVect}{GrVect}
\newcommand{\ab}[1]{{\langle{#1}\rangle}}
\DeclareMathOperator{\AltBilin}{AltBilin}

\newgray{gridline}{0.9}

\begin{document}

\title{Coherence for invertible objects and 
multi-graded homotopy rings}

\author{Daniel Dugger}

\address{Department of Mathematics\\ University of Oregon\\ Eugene, OR
97403}

\email{ddugger@math.uoregon.edu}

\begin{abstract}
We prove a coherence theorem for invertible objects in a
symmetric monoidal category.  This is used to deduce associativity,
skew-commutativity, and related results for multi-graded morphism
rings, generalizing the well-known versions for stable homotopy groups.
\end{abstract}

\maketitle

\tableofcontents

\section{Introduction}
\label{se:intro}

In algebraic topology a classical object of study is the stable
homotopy ring $\pi_*(S)$, which is $\Z$-graded and graded-commutative.
For any topological space (or spectrum) $X$ the stable homotopy groups
$\pi_*(X)$ give a bimodule over $\pi_*(S)$.  The motivation for the
present paper comes from wanting to generalize this basic setup to
more sophisticated homotopy theories, where the homotopy rings and
modules have a more elaborate grading.  Standard examples are the
categories of $G$-equivariant spectra and the category of motivic
spectra.  In these settings it has long been realized that it can be
advantageous to use a grading by an index having to do with the
invertible objects, rather than a grading by integers (which
correspond to integral suspensions/desuspensions of the unit object).
This paper deals with some fundamental questions that arise in this
general situation.

As we explain below, for general grading schema one must take some
care over whether the analog of $\pi_*(S)$ is indeed
associative, and whether the analogs of $\pi_*(X)$ are indeed
bimodules.  Care is also needed in the treatment of
graded-commutativity.  At an even deeper level than these issues, the
ring structure on $\pi_*(S)$ is not exactly canonical---different
choices in the basic setup can result in different isomorphism classes of rings.
We approach these issues by proving a general
coherence theorem for invertible objects in a symmetric monoidal
category; this is the main result of the paper, and is spread across
Theorems~\ref{th:coherence}, \ref{th:coherence2},
\ref{th:multi-coherence1}, and \ref{th:multi-coherence2} below.
After establishing the coherence result we deduce the basic facts
about $\Z^n$-graded homotopy rings as consequences.  

\subsection{Introduction to the problem}

Let $(\cC,\tens)$ be a symmetric monoidal category, and let $S$ denote
the unit.  
Given objects $x_1,\ldots,x_n$ in $\cC$, we write $x_1\tens
\cdots \tens x_n$ as an abbreviation for
\[ x_1 \tens (x_2 \tens (x_3 \tens \cdots (x_{n-1}\tens x_n))).
\] 
We also use $x^k$ as an abbreviation for $x\tens x \tens \cdots \tens
x$ ($k$ factors).  So note that $x^2\tens y^3\tens z$ is an abbreviation for
$(x\tens x) \tens ((y \tens (y\tens y)) \tens z)$,
and that by convention $x^0=S$.  
Finally, for each tuple
$a=(a_1,\ldots,a_n)\in \N^n$, write
\[ \power{x}{a}=x_1^{a_1}\tens \cdots \tens x_n^{a_n}.\]

An object $X$ of $\cC$ is called \dfn{invertible} if there is an
object $Y$ and an isomorphism $\alpha\colon S\ra Y\tens X$.  We will
say that $(Y,\alpha)$ is an inverse for $X$.  In this situation there
turns out to be a unique map $\hat{\alpha}\colon X\tens Y\ra S$ such
that the two evident maps from $(X\tens Y)\tens X$ to $X$ are the
same, and this $\hat{\alpha}$ is an isomorphism (see
Proposition~\ref{pr:alpha-hat} below).  If $a\in \Z$ define
\[ X^a = \begin{cases}
X^a \ (\text{as already defined}) & \text{if $a\geq0$}, \\
Y^{-a}  & \text{if $a<0$.} \\
\end{cases}
\]
Note that given an invertible object $X$, the isomorphism type of $Y$
is uniquely determined; but given a specific choice of $Y$, the map
$\alpha$ is {\it not\/} uniquely determined---it can be varied by an
arbitrary element
of $\Aut(S)$.

Let $X_1,\ldots,X_n$ be a collection of invertible objects in $\cC$,
with inverses $(Y_1,\alpha_1),\ldots,(Y_n,\alpha_n)$.  For $a\in \Z^n$
define
\[ \und{X}^{a}=X_1^{a_1} \tens \cdots \tens X_n^{a_n}.
\]

Assume now that $\cC$ is an additive category and that the tensor
product is an additive functor in each variable.  Let $\pi_*(S)$ be the
$\Z^n$-graded abelian group given by $\pi_{a}(S)=\cC(\um{X}{a},S)$.  
More generally, if $W$ is a fixed object in $\cC$ let $\pi_*(W)$ be the
$\Z^n$-graded abelian group given by
$\pi_{a}(W)=\cC(\um{X}{a},W)$. 
 One of the goals of this paper is the following:

\begin{prop}\mbox{}\par
\label{pr:main1}
\begin{enumerate}[(1)]
\item $\pi_*(S)$ is a $\Z^n$-graded ring, 
\item $\pi_*(W)$ is a $\Z^n$-graded bimodule over $\pi_*(S)$, and
\item There exist elements $\tau_1,\ldots,\tau_n\in
\pi_{(0,\ldots,0)}(S)$ satisfying $\tau_i^2=1$  
such that for all $f\in \pi_a(S)$ and $g\in \pi_b(S)$, where $a,b\in
\Z^n$, one has 
\[ fg=gf\cdot \bigl [ \tau_1^{(a_1b_1)}\cdots \tau_n^{(a_nb_n)}\bigr
].
\]
In fact, $\tau_i$ is just the trace of the identity map on $X_i$ (see
Section~\ref{se:KL} for the definition of trace).  
\end{enumerate}
\end{prop}

\begin{remark}
The groups $\pi_*(W)$ depend on the choice of objects
$X_1,\ldots,X_n$, and therefore we should probably write
$\pi_*^{X}(W)$.  We will always regard the sequence $\und{X}$ as being
understood, however.  Unfortunately, the ring structure from (1)
depends on even more than this: it depends on the choices of
$\alpha_1,\ldots,\alpha_n$.  Given only the sequence $\und{X}$, the
number of isomorphism types of different ring structures is
parameterized by the set $\Aut(S)^n$.  See Proposition~\ref{pr:A}
below.
\end{remark}

To see the difficulty in (1), assume that $f\colon \um{X}{a}\ra S$ and
$g\colon \um{X}{b}\ra S$ are two maps.  Of course we may tensor them
together to form $f\tens g\colon \um{X}{a}\tens \um{X}{b}\ra S\tens
S\iso S$.  However, this only yields an element in
$\pi_{a+b}(S)$ after choosing an isomorphism
$\um{X}{a}\tens \um{X}{b}\iso \um{X}{a+b}$.  The trouble is that there
are many such isomorphisms, and we cannot just choose one at random.
To ensure that $\pi_*(S)$ is associative these isomorphisms must be compatible in
the sense that some evident pentagons all commute.

Both (1) and (2) follow from the fact that one {\it can\/} choose such
isomorphisms in a compatible way.  This is not a particularly hard
result, but it does require some care.  The skew-commutativity in (3)
is more difficult, and when exploring this one quickly realizes the
desirability of a general coherence theory for invertible objects.  This
paper develops such a theory.  

Let us say a little more about skew-commutativity.  Given an
invertible object $X$ and a self-map $f\colon X\ra X$, there is a
well-defined invariant $\tr(f)\in \End(S)$ called the {\it trace\/}
(see Section~\ref{se:KL}).  Define $\tau_X=\tr(\id_X)$ and call this
the \dfn{basic commuter} for the object $X$.  One can prove in this generality that
$\tau_X\in \Aut(S)$ and satisfies $\tau_X^2=\id_S$.  In fact $\tau$
gives a homomorphism $\Pic(\cC)\ra \phantom{\!}_2\Aut(S)$, where
$\Pic(\cC)$ is the group of isomorphism classes of invertible objects
and $\phantom{\,}_2\Aut(S)$ denotes the $2$-torsion elements in
$\Aut(S)$.  This homomorphism is a basic invariant of the symmetric
monoidal category, and governs all commutativity issues.  See
Section~\ref{se:inv} for more information.

The motivating examples for $\cC$ one may wish to keep in mind throughout the paper are:
\begin{itemize}
\item The $G$-equivariant stable homotopy category, where $G$ is a
finite group (or even a compact Lie group).  In this case let
$V_1,\ldots,V_n$ be a collection of finite-dimensional, irreducible,
real representations for $G$ that represent every isomorphism type.
Let $X_i=S^{V_i}$ be the suspension spectra of the one-point
compactifications.
\item The motivic stable homotopy category over some chosen ground
ring.  Here $X_1=S^{1,0}$ and $X_2=S^{1,1}$ are the two basic motivic
spheres.  
\end{itemize}

\begin{remark}
The product  on $\pi_*(S)$ defined above might look different from the
standard composition product that is used for stable homotopy groups.
An easy argument shows that the products are, in fact, the same---see
Remark~\ref{re:comp}.
\end{remark}

\subsection{Coherence results}
Fix an invertible object $X$ with inverse $(Y,\alpha)$.  
Let $w$ be a tensor word in $X$ and $Y$.
As a specific
example, let us look at  the word $w=(X\tens (X\tens Y))\tens (Y\tens X)$.  
Clearly $w\iso X$, but there are different ways to 
construct such an isomorphism.  We might use the chain
\[ w\iso (X\tens S)\tens (Y\tens X)\iso (X\tens S)\tens S \iso X \]
where we used $\hat{\alpha}$ in the first isomorphism and $\alpha$ in
the second.  Or we might use the chain
\[ w\iso (X\tens S)\tens (Y\tens X) \iso X\tens (Y\tens X) \iso
(X\tens Y)\tens X \iso S\tens X \iso X \]
where we have used $\hat{\alpha}$ in both the first isomorphism and in
the fourth.  Are these two composite isomorphisms the same?  Are {\it
  all\/} composite isomorphisms the same?

The answer to the second question depends on how careful we are.  If
we allow ourselves to use the twist isomorphism $t\colon X\tens X\llra{\iso}
X\tens X$ then it is not necessarily true that all composite isomorphisms will be
the same.  However, if we agree not to use the twist then we 
obtain the following result:

\begin{thm}[Coherence without twists]
\label{th:coherence}
Let $w_1$ and $w_2$ be two tensor words in the formal variables $x$
and $y$.   Suppose given
two ``formal composites'' $f,g\colon w_1\ra w_2$, by which we mean
composable sequences of the following kinds of maps:
\begin{enumerate}[(i)]
\item associativity isomorphisms;
\item unital isomorphisms $S\tens W\iso W\iso W\tens S$;
\item $\alpha$ and $\hat{\alpha}$;
\item Maps obtained from the above ones by tensoring with identity
  maps;
\item Inverses of any of the above maps.
\end{enumerate}
Then the maps $f(\cC)$ and $g(\cC)$, obtained by substituting $X$ and $Y$ for $x$
and $y$ and taking the actual composite in $\cC$, are equal.
\end{thm}

\begin{ex}
The awkwardness in the statement of the above proposition is
commonplace in coherence results, because one has to eliminate certain
accidental compositions from occurring.  For example, suppose the
invertible object $X$ happens to be its own inverse: i.e., suppose
$Y=X$.  Then both $\alpha$ and $(\hat{\alpha})^{-1}$ are maps
$S\ra X\tens X$; however, the theorem does not claim that they are the
same map.  Indeed, on a formal level $\alpha$ is a map $S\ra yx$ and
$(\hat{\alpha})^{-1}$ is a map $S\ra xy$, and so there is no choice of
$w_2$ for which we can apply the theorem to these two maps.  
\end{ex}

\begin{ex}
To complement the above ``non-example'' of the proposition, here is a
true application.  Consider
the words $w_1=(x\tens y)\tens x$ and $w_2=x$.  There are two
formal compositions $w_1\ra w_2$ we can construct as follows:
\[ (x\tens y)\tens (x\tens y)\llra{\hat{\alpha}\tens \hat{\alpha}}
S\tens S \llra{\iso} S
\]
and
\[
\xymatrixcolsep{2.7pc}\xymatrix{
 (x y) (xy)\ar[r] &  x(y(xy))\ar[r] &  x ((y x) y)
\ar[r]^-{1\tens \alpha^{-1}\tens 1}
 &  x (S y) \ar[r] &
x y \ar[r]^{\hat{\alpha}} & S.
}
\]
Note that we omitted the tensor symbols in the second composite for
typographical reasons.
The proposition guarantees that the corresponding composites give the
same map in any symmetric monoidal category, for any invertible object
$X$ and inverse $(Y,\alpha)$.
\end{ex}

\begin{remark}[Canonical isomorphisms]
Let $X$ be an invertible object in $\cC$ with inverse $(Y,\alpha)$.
By a ``tensor word'' $w$ in $X$ and $Y$ we can mean either a formal
expression in the symbols ``$X$'' and ``$Y$'' or the actual object 
that results when the expression is evaluated in $\cC$.  We will
usually let the reader deduce the meaning from context, but
occasionally we will write $w(\cC)$ for the latter
interpretation---the evaluation of the formal word $w$ inside of
$\cC$.  Formal tensor words are best thought of as functors into $\cC$
where the allowable inputs are pairs $(X,(Y,\alpha))$.  

Consider the following statement:
given a tensor word $w$ as above, there is a unique $a\in \Z$ for which $w\iso
X^a$.  This is true for formal tensor words, but not necessarily true
for their evaluations in $\cC$.  For example, if our particular object
$X$ is its own inverse ($Y=X$) then we have $X\iso X^1\iso X^{-1}$ and
so the value of $a$ is not unique.  But it is not true that the formal
word ``$X$'' is isomorphic to the formal word ``$X^{-1}$''.

Keeping this nuance of language in mind, we can apply
Theorem~\ref{th:coherence} as follows.  Given a formal word $w$ in $X$
and $Y$, there is a unique $a\in \Z$ for which $w\iso X^a$ (canonical
isomorphism of functors) and moreover the isomorphism can be chosen
from the class described in Theorem~\ref{th:coherence}, in which case
it is {\it canonical\/}.  In this paper such canonical isomorphisms
will always be denoted $\phi$.  The provision of these canonical
isomorphisms is one of the main uses of coherence.
\end{remark}

We will need a coherence theorem that is more sophisticated than
Theorem~\ref{th:coherence}.  To state this, imagine that one has
formal words $w_1,w_2,\ldots,w_n$ in $x$ and $y$ together with a
string of maps
\[ w_1 \llra{f_1} w_2 \llra{f_2} \cdots \lra w_{n-1} \llra{f_{n-1}} w_n.\]
We assume that each $f_i$ is one of the following:
\begin{enumerate}[(i)]
\item an associativity isomorphism;
\item one of the unital isomorphisms $S\tens W\iso W\iso W\tens S$;
\item A twist map $t_{x,x}\colon x\tens x \ra x\tens x$, $t_{x,y}\colon x\tens
  y\ra y\tens x$, $t_{y,x}\colon y\tens x \ra x\tens y$,
  $t_{y,y}\colon y\tens y \ra y\tens y$;
\item Either $\alpha$ or $\hat{\alpha}$;
\item A map obtained from the above ones by tensoring with identity
  maps;
\item An inverse of any of the above maps.
\end{enumerate}
Let $(w,f)$ denote the tuple of $w_i$'s and $f_i$'s.  Define the 
\dfn{parity} of $(w,f)$ to be the total number of times $t_{x,x}$,
$t_{y,y}$, $t_{x,y}$ and $t_{y,x}$ appear---that is, the number of
$i$'s for which one of these maps appears as a tensor factor in $f_i$.

\begin{thm}[Coherence with twists]
\label{th:coherence2}
Let $(w,f)$ and $(w',f')$ be two strings as above, and let $k$ be the
length of the first and $l$ the length of the second.   Assume that
$w_1=w'_1$ and $w_k=w'_l$.  If $(w,f)$ and $(w',f')$ have the same
parity, then the composite of the $f_i$'s is equal to the composite of
the $f'_j$'s in any symmetric monoidal category, when $x$ and $y$ are
replaced with an invertible object $X$ and an inverse $(Y,\alpha)$.  
\end{thm}

\begin{ex}
Consider the
composites
\[ S \llra{\alpha} Y\tens X \llra{t_{Y,X}} X\tens Y
\llra{\hat{\alpha}} S \]
and
\[ S \llra{\phi} (Y\tens Y) \tens (X\tens X) 
\llra{\id \tens t_X} (Y\tens Y)\tens (X \tens X) \llra{\phi^{-1}} S
\]
where $\phi$ is the canonical isomorphism provided by
Theorem~\ref{th:coherence}.  Then Theorem~\ref{th:coherence2}
states that these two composites are the same.  
An attempt to prove this directly will quickly demonstrate the
nontriviality of Theorem~\ref{th:coherence2}.
\end{ex}

Now we turn to coherence theorems involving several different
invertible objects.
Suppose again that $X_1,\ldots,X_n$ are invertible objects in $\cC$.  
For each $i$, let $(X_i^{-1},\alpha_i)$ denote a chosen inverse for
$X_i$.
Let $w$ be a tensor word in $X_1,\ldots,X_n$ and
$X_1^{-1},\ldots,X_n^{-1}$.  It is clear that $w$ is formally
isomorphic to $\power{X}{a}$ for
a uniquely determined $a\in \Z^n$.  We want a result which says
that different ways of constructing such an isomorphism yield the same
result.  

\begin{remark}
In our statements of the next two results we dispense with the
phrasing about formal compositions and their instances inside of a
given symmetric monoidal category.  However, this language should be
taken as implicit in the statements.
\end{remark}

\begin{thm}[Coherence without self-twists, multi-object case]
\label{th:multi-coherence1}
Let $w$ be a tensor word in the symbols $X_i$ and $X_i^{-1}$, $1\leq
i\leq n$.
There is an isomorphism $w\iso \power{X}{a}$ constructed as a
composite of the following kinds of maps, and moreover this isomorphism
is unique.  The maps we are allowed to use are
\begin{enumerate}[(i)]
\item associativity isomorphisms;
\item unital isomorphisms;
\item commutativity isomorphisms $X_i\tens X_j \ra X_j\tens X_i$ and
  $X_i\tens X_j^{-1} \ra X_j^{-1}\tens X_i$ for $i\neq j$;
\item the maps $\alpha_i$ and $\hat{\alpha}_i$;
\item maps obtained from (i)--(iv) by tensoring with identities;
\item all inverses of maps in (i)--(v).
\end{enumerate}
\end{thm}

We also have a more general version involving parity checks.  Suppose
$w_1,w_2,\ldots,w_k$ are tensor words in the $X_i$'s and $X_i^{-1}$'s,
and consider a composite
\[ w_1 \llra{f_1} w_2 \llra{f_2} \cdots \lra w_{k-1} \llra{f_{k-1}}
w_k.
\]
We assume that each $f_i$ is one of the following:
\begin{enumerate}[(i)]
\item an associativity isomorphism;
\item one of the unital isomorphisms $S\tens W\iso W\iso W\tens S$;
\item A twist map $t_{X_i,X_j}\colon X_i\tens X_j \ra X_j\tens X_i$, 
$t_{X_i,X_j^{-1}}\colon X_i\tens
  X_j^{-1}\ra X_j^{-1}\tens X_i$, $t_{X_i^{-1},X_j}\colon X_i^{-1}\tens X_j \ra X_j\tens X_i^{-1}$,
 or  $t_{X_i^{-1},X_j^{-1}}\colon X_i^{-1}\tens X_j^{-1} \ra X_j^{-1}\tens 
X_i^{-1}$, where possibly $i=j$;
\item One of the $\alpha_i$'s or $\hat{\alpha}_i$'s;
\item A map obtained from the above ones by tensoring with identity
  maps;
\item An inverse of any of the above maps.
\end{enumerate}
Define the \mdfn{$i$-parity} of the string $(w,f)$ to be the total number of
times $t_{X_i,X_i}$, $t_{X_i,X_i^{-1}}$, $t_{X_i^{-1},X_i}$, and
$t_{X_i^{-1},X_i^{-1}}$ appear in the $f_j$'s.  We have the following:

\begin{thm}[Coherence with self-twists, multi-object case]
\label{th:multi-coherence2}
Let $(w,f)$ and $(w',f')$ be two strings as above, where the length of
the first is $k$ and the length of the second is $l$.  Assume
$w_1=w'_1$ and $w_k=w'_l$.  If $(w,f)$ and $(w',f')$ have the
same $i$-parity for all $1\leq i\leq n$, then the composites of the
two strings are the same map.
\end{thm}

\subsection{Applications}
The first application answers the question raised at the beginning of
the paper.  If $f\in \pi_{a}(S)$ and $g\in \pi_{b}(S)$
then form the tensor product $f\tens g\colon \um{X}{a}\tens
\um{X}{b}\ra S\tens S\iso S$.  Theorem~\ref{th:multi-coherence1}
supplies a canonical isomorphism $\um{X}{a}\tens \um{X}{b} \ra \um{X}{a+b}$,
and using this we obtain an element $f\cdot g\in \pi_{a+b}(S)$.
Similarly, one obtains maps $\pi_{a}(S)\tens \pi_{b}(W) \ra
\pi_{a+b}(W)$ and so forth.  Coherence guarantees that these pairings
all have the desired associativity (see Section~\ref{se:MacLane} for details).  

Given a map $f\colon \power{X}{a} \ra \power{X}{b}$ there are two
evident ways to recover an element of $\pi_{a-b}(S)$.  We can tensor
on the left with $\power{X}{-b}$ and then use the canonical
isomorphisms from Theorem~\ref{th:multi-coherence1},
or we can tensor on the right and use canonical isomorphisms.  We call
the associated elements $[f]_r$ and $[f]_l$, respectively.  Another
application of coherence is to relate these two elements:
\begin{myequation}
\label{eq:r=l}
 [f]_r=[f]_l\cdot \prod \tau_i^{b_i(a_i-b_i)}
\end{myequation}
where the $\tau_i$'s are the basic commuters of the $X_i$'s.  
This and many related formulas are developed in
Section~\ref{se:app}.  

Let us very briefly indicate the idea behind skew-commutativity.  If
$f\colon \power{X}{a}\ra S$ and $g\colon \power{X}{b}\ra S$ then we
may form the diagram
\[ \xymatrix{
\power{X}{a}\tens \power{X}{b}\ar[d]_{t_{a,b}} \ar[r]^{f\tens g} & S\tens
S\ar@{=}[r] & S. \\
\power{X}{b}\tens \power{X}{a} \ar[ur]_{g\tens f}
}
\]
A little work gives that inside $\pi_*(S)$ we have $f\cdot g=g\cdot
f\cdot [t_{a,b}]_r$ (note that by (\ref{eq:r=l})
one has $[t_{a,b}]_r=[t_{a,b}]_l$). 
The content of Proposition~\ref{pr:main1}(c) is the
identification of $[t_{a,b}]_r$ as a product of basic commuters;
this is a direct consequence of Theorem~\ref{th:multi-coherence2}, which says that the
associated composite $S\ra S$ is determined purely by the parities
involved.  See Section~\ref{se:app} for complete details.

\subsection{The stable motivic homotopy ring}
We close this long introduction with a very specific example.  In the
stable motivic homotopy category (over a chosen ground field) there
are two basic spheres denoted $S^{1,0}$ and $S^{1,1}$.  These are
invertible objects.  More generally one sets
$S^{p,q}=(S^{1,0})^{\Smash (p-q)}\Smash (S^{1,1})^{\Smash(q)}$, for
any $p,q\in \Z$.  The bigraded stable homotopy ring $\pi_{*,*}(S)$ is
an instance of the general situation considered in this paper,
although unfortunately the bigrading is different from the generic
bigrading we adopted for Proposition~\ref{pr:main1}: the motivic group
$\pi_{a,b}(S)$ corresponds to what we have been calling $\pi_{(a-b,b)}(S)$.

The basic commuter for $S^{1,0}$ is the element $-1\in \pi_{0,0}(S)$.  The
basic commuter for $S^{1,1}$ is represented by the twist map
$S^{1,1}\Smash S^{1,1}\ra S^{1,1}\Smash S^{1,1}$; in motivic homotopy
theory it is usually denoted $\epsilon \in \pi_{0,0}(S)$.  The 
skew-commutativity result for the motivic stable homotopy ring is the
following, obtained as a direct corollary of Proposition~\ref{pr:main1}:

\begin{prop}
\label{pr:motivic-commute}
For $f\in \pi_{a,b}(S)$ and $g\in \pi_{c,d}(S)$ one has
\[ fg=gf\cdot (-1)^{(a-b)(c-d)}\cdot \epsilon^{bd}.
\]
\end{prop}

Now assume that the ground field is $\C$, so that there is a
realization map $\psi$ from
the stable motivic homotopy category to the classical stable homotopy
category of topological spaces.  This induces a collection of group
homomorphisms $\psi_{p,q}\colon \pi_{p,q}(S) \ra \pi_p(S)$.  One's
first guess might be that
these maps assemble into a ring homomorphism
$\psi\colon \pi_{*,*}(S)\ra \pi_*(S)$, but this is not quite right.  Instead
there is the following identity:

\begin{prop}
\label{pr:motivic-compare}
For $f\in \pi_{a,b}(S)$ and $g\in \pi_{c,d}(S)$ one has
\[ \psi(fg)=\psi(f)\cdot\psi(g) \cdot (-1)^{b(c-d)}.
\]
\end{prop}

This result was one of the motivations for the work in this paper;
we include the proof as a brief
appendix.  

\begin{remark}
It is satisfying to check that Propositions~\ref{pr:motivic-commute} and
\ref{pr:motivic-compare} are (taken together) compatible with the 
graded-commutativity of the classical stable homotopy ring.  This uses
that $\psi(\epsilon)=-1$.  
\end{remark}

\vspace{0.1in}

\subsection{Generalizations}
Let $\Pic(\cC)$ be the group of isomorphism classes of invertible
objects in $\cC$.  Let $A$ be an abelian group and let $h\colon A\ra
\Pic(\cC)$ be a homomorphism (the case $A=\Pic(\cC)$ is the main one
of interest, but it is useful to work in slightly greater generality).  
The question we pose is whether $\pi(S)$ can be regarded as
an $A$-graded ring.  We can certainly choose, for each $a\in A$, an
object $X_a$ in the isomorphism class $h(a)$.  We can then
define
an $A$-graded abelian group by
\[ \pi^A_a(S)=\cC(X_a,S). \]  
To give a pairing on this graded group one should start by choosing isomorphisms
$\sigma_{a,b}\colon X_{a+b}\ra X_a\tens X_b$, and for the pairing to be associative
these isomorphisms must satisfy a certain compatibility condition (a unital
condition should also be imposed).  Our coherence results show that
when $A$ is finitely-generated and free this can be
accomplished---although it is important to realize that the method for
doing so is not quite canonical, depending both on a choice of basis
for $A$ and a choice of the $\alpha_i$ maps we encountered earlier.
What about other values of $A$?  We will show that
\begin{enumerate}[(1)]
\item For any abelian group $A$, the isomorphisms $\sigma_{a,b}$ can
be chosen so that $\pi_*^A(S)$ is an associative and unital ring;
\item However, the choices involved in (1) are not canonical and the
different isomorphism classes of rings one can obtain are in bijective
correspondence with the elements of the group cohomology $H^2(A;\Aut(S))$.
\end{enumerate}
In homotopy theory one
often hears the slogan ``one should grade things by the invertible
objects''; point (2) above suggests that this is a little more dicey
than one might wish.
These results are in Section~\ref{se:schema}.

\subsection{Some background, and an apology}  From a certain
perspective this paper is entirely on the subject of getting signs
correct---although the ``signs'' are not just $\pm 1$ but rather
elements in the $2$-torsion of $\Aut(S)$, where $(\cC,\tens,S)$ is a
symmetric monoidal category.  Many of the results are undoubtedly
folklore, but just lacking a convenient reference.  Since this is a
subject where it seems particularly important to {\it have\/} a convenient
reference---no one likes to think about signs---we have included quite
a bit of exposition (perhaps overdoing it on occasion).

Associativity and commutativity for $RO(G)$-graded stable homotopy
rings have typically been dealt with in a somewhat different way than
what we describe here.  In essence, the various choices for
isomorphisms are built into the framework from the very beginning, and
one is tasked with keeping track of them.  We refer the reader to
\cite[Chapter 13, Section 1]{Ma} and
\cite[Section 21.1]{MS} for detailed discussions.  

Certainly  Proposition~\ref{pr:main1} and related results are if not
well-known then at least not surprising---although I
wonder if the sign in the multiplicativity of the forgetful map (see
Proposition~\ref{pr:motivic-compare}) has been noticed before, either in
the motivic or equivariant context.  I have also been unable to find a
reference in the literature similar to 
Proposition~\ref{pr:l-r-general}, even though the sign questions
dealt with by that result are ubiquitous.

Invertible objects are well-studied in the literature (for example, in
\cite{FW}), but in somewhat sporadic places---and there seem to be some gaps.  
For a self-map $f\colon X\ra X$ where $X$ is invertible, there are two
ways to obtain an element of $\End(S)$.  One is called the {\it
trace\/} of $f$, and the other is something that does not have a
standard name---in this paper we call it the {\it $D$-invariant\/}.  These
two invariants can be different, although they sometimes get confused.
We attempt to give a careful treatment in Section~\ref{se:inv}.

As far as the coherence statements are concerned, the earliest result
along these lines seems to be \cite[Lemma 1.4.3]{D}, due to Deligne.
However, Deligne's result (stated without proof) only applied to
symmetric monoidal categories where the self-twists $X\tens X\ra
X\tens X$ are all equal to the identity; as is clear from the results
listed above, this omits the important and nontrivial phenomena that
occur in the general case.  Symmetric monoidal categories in which all
objects are invertible are treated again in \cite{FW}.  A
classification theorem is given (see \cite[Corollary 6.6]{FW}), from
which coherence results are easily deducible, but again only in the
case where the self-twists are all equal to the identity.  Another
sort of classification theorem for such categories is given in the
unpublished PhD thesis \cite[Chapter II, Section 2, Proposition 5]{H};
but although Ho\`ang's theorem allows for non-identity self-twists the
classification  is of a different nature and does not seem to yield any
coherence results.  The excellent and influential paper \cite{JS} has
coherence results in the braided context, but not for invertible
objects; it gives a classification theorem for braided monoidal categories
where all objects are invertible,  but again not yielding coherence
results in any evident way.  The literature contains many more
sophisticated coherence results than the ones presented here, and so
it seems to be merely an unfortunate accident that there is no
convenient reference for them.

\subsection{Organization of the paper}

In Section~\ref{se:MacLane} we give a brief review of MacLane's
coherence theorem for monoidal categories, and explain how it gives
rise to associativity results for $\N$-graded morphism groups.
Section~\ref{se:KL} then reviews the deeper coherence theorem of
Kelly-Laplaza, which applies to symmetric monoidal categories with
left duals.  Section~\ref{se:inv} develops the basic theory of
invertible objects, in particular establishing that the trace of the
identity map on such an object has order at most two; this is a key
result used throughout the paper.  In Section~\ref{se:coherence} we
prove the coherence theorems for invertible objects, and in
Section~\ref{se:app} we give the applications to $\Z^n$-graded morphism
rings.  Finally, Section~\ref{se:schema} deals with the topic of grading
morphism rings by non-free abelian groups.  

\subsection{Acknowledgements}

This research was partially supported by NSF grant DMS-0905888.  The author
is grateful to Sharon Hollander, Peter May, Victor Ostrik, and Vadim Vologodsky for
helpful conversations.


\section{Review of MacLane's coherence}
\label{se:MacLane}

Here we recall how MacLane's classical coherence theorem for symmetric
monoidal categories gives rise to an associativity result for
$\N^n$-graded morphism rings.

\medskip

Let $(\cC,\tens,S)$ be a symmetric monoidal category.  
Where needed, we will denote the associativity isomorphism
$(x\tens y)\tens z\iso x\tens(y\tens z)$ by $a$ and the symmetry
isomorphism $x\tens y\iso y\tens x$ by $t$.

Let $w$ be any tensor word made up of formal variables $x_i$---for 
instance, the word
$((x_1\tens x_2)\tens x_1) \tens (x_2\tens x_3)$ is one example.  Such
tensor words can be identified with certain kinds of functors
$\cC^n\ra \cC$, where $n$ is the number of letters in the word.  Precisely, these
are the functors that can be built up from $\tens\colon\cC^2\ra \cC$
using composition and the operation of cartesian
product with identity  maps.  

Using the associativity and commutativity isomorphisms, there is a
``formal isomorphism'' (or natural isomorphism) between $w$ and some
word $\power{x}{a}$, for a uniquely determined $a\in \N^n$.  To fix such an
isomorphism, here is what we do.  First, relabel the $x_1$'s in $w$ as
$x_{1a},x_{1b},x_{1c}$, etc., with the indices appearing
alphabetically from left to right in the word.  Do the same for all
the other $x_i$'s.  Let $w'$ denote the new word thus constructed.
Regard $w'$ as a functor $\cC^N\ra \cC$, where $N$ is the total number
of variables in $w'$.  Using the associativity, commutativity, and unital
isomorphisms one can construct a natural isomorphism between the
functor corresponding to $w'$ and the functor corresponding to the
word
\[ (x_{1a}\tens x_{1b}\tens \cdots ) \tens (x_{2a}\tens x_{2b}\tens
\cdots) \tens \cdots \tens (x_{na}\tens x_{nb}\tens \cdots).
\]
Moreover, MacLane's coherence theorem \cite[Theorem XI.1.1]{M} 
says
that this natural isomorphism is {\it uniquely determined\/}.  The
exact choices of associativity and commutativity isomorphisms used to
construct it are definitely not unique, but the composite isomorphism
itself is unique.  Now evaluate this natural isomorphism in the case
where all the $x_{1,*}$ objects are equal to $x_1$, all the $x_{2,*}$
objects are equal to $x_2$, etc.  This is our definition of
$\phi\colon w\llra{\iso} \power{x}{a}$

Using the observation of the last paragraph, for $a,b\in \N^n$ we obtain canonical
isomorphims
\[ \phi_{a,b}\colon \power{x}{a} \tens \power{x}{b} \llra{\iso} \power{x}{a+b}
\]
such that the following diagram commutes:
\begin{myequation}
\label{eq:assoc}
\xymatrix{
(\power{x}{a}\tens \power{x}{b}) \tens \power{x}{c}
  \ar[rr]^\iso\ar[d]_{\phi_{a,b}\tens \id}
&& \power{x}{a}\tens (\power{x}{b} \tens \power{x}{c})
  \ar[d]^{\id\tens \phi_{b,c}} \\
\power{x}{a+b}\tens \power{x}{c}\ar[dr]_{\phi_{a+b,c}} && \power{x}{a} \tens
\power{x}{b+c}\ar[dl]^{\phi_{a,b+c}}\\
& \power{x}{a+b+c}.
}
\end{myequation}
The reason it commutes is again by the coherence theorem.  Both ways
of moving around the diagram are instances of a natural transformation
made up of the associativity and commutativity isomorphisms, where one
relabels the $x_1$'s appearing as $x_{1a},x_{1b}$, etc., and the same
for the other variables.  MacLane's theorem says that there is a unique such
natural transformation, and so the two ways of moving around the
diagram must be the same.    (Note also that when $a$ or $b$ is the
zero vector then $\phi_{a,b}$ is the unital isomorphism from the
symmetric monoidal structure).

Now assume that $\cC$ is also an additive category, and that the
tensor product is an additive functor in each variable.  
Let $X_1,\ldots,X_n$ be fixed objects in $\cC$.  Consider the $\N^n$-graded
abelian group
\[ R=\oplus_{a\in \N^n} \cC(\power{X}{a},S).
\]
We will also write $R_{a}$ for $\cC(\power{X}{a},S)$.     
We claim that $R$ has the structure of an $\N^n$-graded ring.  The product is
defined as follows.  If $f\in R_{a}$ and $g\in R_{b}$,
define $fg\in R_{a+b}$ to be the composition
\[ \power{X}{a+b}\llra{\phi_{a,b}^{-1}} \power{X}{a}\tens \power{X}{b}
\llra{f\tens  g}
S\tens S \iso S.
\]

\begin{remark}
\label{re:comp}
Note that the above product can also be described as the composition
\[ \power{X}{a+b}\iso 
\power{X}{a}\tens\power{X}{b} \llra{\id\tens g} \power{X}{a}\tens S
\llra{f\tens \id} S\tens S \iso S.
\]
In this way the product in $R$ can be thought of as induced by the
composition in the category $\cC$: $fg$ comes from composing $f$ with an
appropriately ``suspended'' version of $g$.
\end{remark}

\begin{prop}
\label{pr:R}
$R$ is a graded ring (associative and unital), and $R_0$ is central.
\end{prop}

\begin{proof}
Distributivity follows immediately from the fact that $\tens$ is
biadditive: for instance, if $f,g\in R_{a}$ and $h\in
R_{b}$ then the map $(f+g)\tens h$ is equal to $(f\tens
h)+(g\tens h)$.  So the same remains true when we precompose both with
$\phi_{a,b}^{-1}$.  

Associativity follows, by an easy argument, from the fact that diagram
(\ref{eq:assoc}) is commutative.  The fact that $\phi_{a,b}$ equals the
unital isomorphism when $a$ or $b$ is zero implies that the identity 
element $\id_S\in R_0$ is a unit for $R$.

For the centrality of $R_0$,
let $f\colon S\ra S$ and let $g\colon \power{X}{a}\ra S$. 
The following
diagram is commutative:
\[ \xymatrixrowsep{1pc}\xymatrix{
& \power{X}{a}\tens S \ar[dd]^t\ar[r]^{g\tens f}
 & S\tens S \ar[dd]^t\ar[dr]^\iso \\
\power{X}{a}\ar[ur]^\iso\ar[dr]_\iso &&& S \\
& S\tens \power{X}{a} \ar[r]_{f\tens g} & S\tens S \ar[ur]_\iso
}
\]
The composition across the top is $f\cdot g$, and 
across the bottom is $g\cdot f$.  
\end{proof}

\begin{remark}
This section serves as the prototype for what will happen in the rest
of the paper.  In the case where the objects $X_i$ are invertible
under the tensor product, we wish to extend $R$ to a $\Z^n$-graded
ring.  This requires extending the construction of the $\phi_{a,b}$'s,
which in turn depends on a more sophisticated version of coherence.
\end{remark}

\section{Kelly-Laplaza coherence}
\label{se:KL}
Let $\cU$ be a set.  Kelly and Laplaza \cite{KL} describe the ``free symmetric
monoidal category with left duals'' on the set $\cU$, denoted here
$\KL(\cU)$.  In this section we review this
construction.

\medskip
\subsection{Preliminaries}
\label{se:KL-pre}
Let $(\cC,\tens,S)$ be a symmetric monoidal category, and let $X$ be
an object.  Recall that a \dfn{left dual} for $X$ is an object $Y$
together with maps $\alpha\colon S \ra Y\tens X$ and
$\hat{\alpha}\colon X\tens Y \ra S$ such that the composites
\[ X=X\tens S \llra{\id\tens \alpha} X\tens Y \tens X
\llra{\hat{\alpha}\tens \id} S\tens X=X \]
and
\[ Y=S\tens Y \llra{\alpha\tens \id} Y\tens X\tens Y \llra{\id\tens
\hat{\alpha}} Y\tens S=Y
\]
are the respective identities (we are not bothering to write the
associativity isomorphisms in the composites, even though they are
there).  To give an object $Z$ the structure of a left dual of $X$ is
the same as giving the functor $Z\tens (\blank)$ the structure of a
right adjoint to $X\tens (\blank)$.  This observation makes it clear
that if $(Y,\alpha,\hat{\alpha})$ and $(Y',\alpha',\hat{\alpha}')$ are
both left duals for $X$ then there is a unique isomorphism $Y\ra Y'$
that is compatible with the extra structure.

\begin{defn}
A \dfn{symmetric monoidal category with left duals} is a symmetric monoidal
category $(\cC,\tens,S)$ together with an assignment $X\mapsto
(X^*,\alpha_X,\hat{\alpha}_X)$ that equips every object of $\cC$ with a
left dual. 
 (Warning: Note that $(X^*)^*$ need not equal $X$,
although they will be isomorphic).  
\end{defn}

In the above setting, there is a unique way of making $X\ra X^*$ into
a contravariant functor.  This is not included as part of the
definition only to minimize the number of things that need to be
checked in applications.  We will not need the functoriality of duals.

It is easy to see that any closed symmetric monoidal category has left duals.

Suppose $X$ has a left dual and $f\colon X\ra X$ is a map.  Then we
may form the composite
\[ S \llra{\alpha} X^*\tens X \llra{\id\tens f} X^*\tens X \llra{t} X\tens X^* \llra{\hat{\alpha}} S,
\]
and this composite is called the \dfn{trace} of $f$.  The uniqueness
of left duals (up to isomorphism) shows that $\tr(f)$ is not dependent on the
choice of left dual for $X$.

The trace satisfies the following properties:

\begin{prop} Let $\cC$ be a symmetric monoidal category with left duals.
\label{pr:trace}
\begin{enumerate}[(a)]
\item If there is a commutative diagram
\[ \xymatrix{X\ar[r]^q\ar[d]_f & Z\ar[d]^g \\
X\ar[r]^q & Z
}
\]
in which $q$ is an isomorphism, then $\tr(f)=\tr(g)$.
\item If $f\colon X\ra Y$ and $g\colon Y\ra X$ then $\tr(fg)=\tr(gf)$.  
\end{enumerate}
\end{prop}

\begin{proof}
For part (a), observe that if $(X^*,\alpha,\hat{\alpha})$ is a  left dual for $X$ then
$X^*$ is also a left dual for $Z$ via the maps $\alpha'=(\id_{X^*}\tens
q) \alpha$ and $\hat{\alpha}'=\hat{\alpha}(q^{-1}\tens \id_{X^*})$.
Using this, (a) is an easy exercise.

Part (b) is much harder.  It is not needed in the present paper, but
included for expository purposes.  For a proof, see \cite[Proposition 2.4]{PS}.
\end{proof}

We will also need the following fundamental result:

\begin{lemma}
\label{le:aut-ab}
In a symmetric monoidal category the monoid $\End(S)$ is abelian.
\end{lemma}

\begin{proof}
Suppose $f,g\colon S\ra S$ and consider the following commutative
diagram:
\[ \xymatrix{
S\tens S \ar[r]^{f\tens \id}\ar[d] & S\tens S \ar[r]^{\id\tens g}\ar[d]
& 
S\tens S \ar[d] \\
S \ar[r]^f & S \ar[r]^g & S.
}
\]
This shows that $gf$ is the composite
\[ S\llra{\iso} S\tens S \llra{f\tens g} S \llra{\iso} S.
\]
A similar diagram also shows that $fg$ is also equal to this
composite, so $gf=fg$. 
\end{proof}

\subsection{The construction}
For any category $\cU$ the paper \cite{KL} constructs the free
symmetric monoidal category with left duals on $\cU$.   We will only
need this construction where $\cU$ only has identity maps---i.e.,
$\cU$ is just a set.  See Remark~\ref{re:KL-general}, however, for hints about the
general case.

Note that for every element
$X\in \cU$ our category must have an identity map $\id_X$ and
therefore a self-map of $S$ obtained by taking the trace.  These
traces will all need to commute, since all self-maps of $S$
commute by Lemma~\ref{le:aut-ab}.  
So 
let $\N\langle \cU\rangle$ be the free commutative monoid on the set
$\cU$.  If $X\in \cU$ we think of the element $[X]\in \N\langle
\cU\rangle$ as the formal trace of the identity map on $X$.  Our construction for $\KL(\cU)$ will have $\N\langle
\cU\rangle$ as its set of self-maps of $S$.

Define a \dfn{signed set} to be a set $A$ together with a
function $\tau\colon A\ra \{+,-\}$.  If $A$ is a signed set let $A^*$
be the same set but with the signs reversed.  If $A$ and $B$ are
signed sets then $A\amalg B$ denotes the disjoint union with the
evident signs.  A \dfn{bipartition} of a signed set $A$ is a
directed graph with $A$ as the vertex set, having the properties that
\begin{enumerate}[(i)]
\item The tail of every edge is marked with $-$ and the head of every
edge is marked with $+$;
\item Every element of $A$ is a vertex of exactly one edge.  
\end{enumerate}
If $A$ and $B$ are signed sets, then a \dfn{correspondence} from $A$ to $B$
is a bipartition of $A^*\amalg B$.  One can make a picture of such a
thing by drawing the elements of $A$ on one ``level'', the elements of
$B$ on a lower level, and then drawing the edges of the bipartition.
For example:
\[ \xymatrix{
A: & + \ar[dr] & +\ar[dl] & + \ar@/_3ex/[r]  & -  & - \\
B: & +  & +  & -\ar[urr]  & -\ar@/^3ex/[r] & +
}
\]
Note the convention for drawing edges with vertices at the same level:
if the vertices are on the top level we draw a cup $\cup$, and if the
vertices are on the bottom level we draw a cap $\cap$.  Also, there is
a simple technique for getting the direction of the arrows straight: each
element with sign $+$ should be pictured as a small downward arrow
$\downarrow$, and elements with sign $-$ are pictured as small upward
arrows $\uparrow$.  These small arrows must join (compatibly) with the
edges in the correspondence.   Finally, note again that the data in these
pictures is really just ``what connects to what''.  The exact physical
paths of the arrows in the picture are irrelevant, only where the
arrows begin and end.  

Given a correspondence  from $A$ to $B$, and a correspondence
 from $B$ to $C$, we may compose these to get a correspondence
from $A$ to $C$.  This is best described in terms of the pictures: one
stacks the pictures on top of each other and composes the edges
head-to-tail as expected.  Note that there might be extra ``loops'' in
the picture, and these must be discarded.  For example, the composition
\[ 
\xymatrix{
A: & + \ar[dr] & +\ar[dl] & + \ar@/_3ex/[r]  & -  & - \\
B: & +\ar[dr]  & +\ar@/_3ex/[r]  & -\ar[urr]  
& -\ar@/^3ex/[r]\ar@/_8ex/[rrr] & + \ar@/_3ex/[r]  & -\ar@/^3ex/[r] & +\\
C: & & + 
} 
\]
equals the correspondence
\[
\xymatrix{
A: & + \ar@/_6ex/[rrrr] & +\ar[d] & + \ar@/_3ex/[r]  & -  & - \\
C: & & + 
} 
\]

\medskip

We are ready to define the category $\KL(\cU)$.  
An object will be a formal word made from the set $\cU$ and the special
symbol $S$ using tensors and duals: e.g., $w=((X^*)^*\tens S)\tens ((X\tens
Y)^*\tens (Y^*\tens Z))$.  To each such word we associate its
underlying set of letters $P(w)$, together with a function
$\tau\colon P(w)\ra \{+,-\}$.  In the above example
$P(w)=\{X_1,X_2,Y_1,Y_2,Z\}$ (the indices distinguish the different
occurrences of the letters in the word) and the sign function has
$\tau^{-1}(+)=\{X_1,Z\}$, $\tau^{-1}(-)=\{X_2,Y_1,Y_2\}$.  In general
$P$ is defined inductively by setting $P(X)=\{X\}$ if
$X\in \cU$, $P(S)=\emptyset$, $P(u\tens
v)=P(u)\amalg P(v)$, and $P(u^*)=P(u)^*$.  
Note that there is an evident map $P(w)\ra \cU$  that sends each
formal symbol to the corresponding element of $\cU$.  

Let $w_1$ and $w_2$ be two formal words.  We define a map from $w_1$
to $w_2$ to be
a pair $(\theta,\lambda)$ where $\theta$ is a bipartition of
$P(w_1)^*\amalg P(w_2)$ for which the head and tail of every edge are
sent to the same object of $\cU$, and where $\lambda$ is an element of
$\N\langle \cU \rangle$.
Given a map $(\theta,\lambda)\colon w_1\ra w_2$ and $(\phi,\mu)\colon
w_2\ra w_3$, the composite is $(\phi\theta,\lambda+\mu+\sum_i [X_i])$
where $\phi\theta$ is the composition of correspondences and where the
$X_i$'s are the objects labelling each of the loops that was discarded
during the composition process.  Said differently, every loop in which
the vertices were labelled by an object $X\in \cU$ contributes a
factor of $\tr(\id_X)=[X]$ to the composition.

We may again depict maps in $\KL(\cU)$ via pictures.  For
example, here is a map from $(X^*)^*\tens (X\tens Y)^* \tens (Z\tens
Y)$ to $Z$:
\[\xymatrix{
X_+\ar@/_3ex/[r] & X_- & Y_- & Z_+\ar[dl] & Y_+\ar@/^3ex/[ll] \\
&& Z_+
}
\]  
Note that the precise word forming the domain (or codomain) of the map
is not retrievable from the picture; that is, the picture only shows
$P(w)$ together with a linear ordering, not $w$ itself.  This is
actually a feature rather than a bug!   If two words differ only in
the placement of parentheses, for example, notice that there is a
canonical isomorphism between them.  Similarly, observe that
$(w^*)^*$ is canonically isomorphic to $w$, and $(w_1\tens w_2)^*$ is
canonically isomorphic to $w_1^*\tens w_2^*$.

The category $\KL(\cU)$ is a symmetric monoidal category with left
duals.  We leave the reader the (not difficult, but informative)
exercise of checking this and identifying the necessary structures.  
The main result of \cite{KL} is the following:

\begin{thm}[Kelly-Laplaza Coherence Theorem]
The category $\KL(\cU)$ is the free symmetric monoidal category with
left duals on the set $\cU$.  
\end{thm}

\begin{remark}
\label{re:KL-general}
The paper \cite{KL} actually describes the free symmetric monoidal
category with left duals on a {\it category\/} $\cA$.  We have only
discussed the case where $\cA$ is discrete (i.e., only has identity
maps) because this is all we need for our present purposes.  The
general case is not very different, however.  A map in $\KL(\cA)$ is a
correspondence equipped with a labelling of the edges by maps in
$\cA$, having the property that if an edge has head $Y$ and tail $X$
then the label belongs to $\Hom_\cA(X,Y)$.  When composing labelled
correspondences one composes the labels in the evident manner.
Finally, the monoid of formal traces $\N\langle \cU\rangle$ must be
replaced by something more complex: every self-map in $\cA$ must have
a formal trace, and these must satisfy the cyclic property of
Proposition~\ref{pr:trace}(b).  It is easy to write down the
universal monoid  having these properties; see \cite{KL} for details.
\end{remark}

\subsection{Uses of coherence}
Now let $\cC$ be a symmetric monoidal category with left duals.  Let $\cU\subseteq
\ob(\cC)$ be a set of objects.  
The
Kelly-Laplaza Coherence Theorem says that there is a map of
symmetric monoidal categories
$F\colon \KL(\cU) \ra \cC$ sending the formal word $[X]$ to $X$, and
the formal word $[X]^*$ to $X^*$, for
each $X\in \cU$.  The functor $F$ also has the following behavior:

\begin{align*}
 &\xymatrix{ X_- \ar@/^3ex/[r] & X_+} \quad \text{is sent to} \quad 
S\llra{\alpha} X^*\tens X\\[0.2in]
 &\xymatrix{ X_+ & X_-\ar@/_3ex/[l] } \quad \text{is sent to} \quad
S\llra{\alpha} X^*\tens X\llra{t} X\tens X^*\\[0.2in]
 &\xymatrix{ X_+ \ar@/_3ex/[r] & X_-} \quad \text{is sent to} \quad 
X\tens X^* \llra{\hat{\alpha}} S\\[0.2in]
 &\xymatrix{ X_- & X_+\ar@/^3ex/[l] } \quad \text{is sent to} \quad
X^*\tens X \llra{t} X\tens X^* \llra{\hat{\alpha}} S.
\end{align*}

\vspace{0.1in}

\noindent
[Note: It is worth taking time to think about the second and fourth
cases; these composites are also represented by other pictures in which
there is a crossing between the edges, but in $\KL(\{\cU\})$ such pictures represent the
same maps as what we have given---remember that the only part of the
pictures that matters is ``what connects to what''.]

As an example, suppose that $X$ and $Y$ are dualizable objects in
$\cC$ with chosen duals $X^*$ and $Y^*$.  Consider the following two
maps from $X\tens Y\tens X^*\tens Y^*\tens Y$ to $Y$:
\begin{myequation}
\label{eq:comp1} XYX^* Y^* Y \ra YXX^* Y^* Y \ra YSY^*Y=YY^*Y\ra SY=Y 
\end{myequation}
and
\begin{myequation}
\label{eq:comp2}
 XYX^*Y^*Y \ra XX^* YY^*Y \ra SYS =Y
\end{myequation}
(we have suppressed the tensor symbols and associativity maps; each of
the displayed maps is the evident one that uses the symmetric monoidal structure
and the duality maps).  
Are the maps in (\ref{eq:comp1}) and (\ref{eq:comp2}) 
guaranteed to be the same in $\cC$?  We work in
$\KL(\{X,Y\})$ and note that the two composites are represented by the
following pictures:
\[ \xymatrixrowsep{1pc}\xymatrixcolsep{1pc}\xymatrix{
X_+ \ar[dr] & Y_+\ar[dl] & X_- & Y_- & Y_+\ar[d] &  
&& X_+\ar[d] & Y_+\ar[dr] & X_- & Y_- & Y_+\ar[d] \\
\bullet\ar[d] & \bullet \ar@/_3ex/[r] & \bullet\ar[u] & \bullet\ar[u]
& \bullet\ar[d] 
&&& \bullet \ar@/_3ex/[r]&\bullet\ar[ur] & \bullet\ar[d] &\bullet\ar[u]
&\bullet\ar@/^3ex/[l]
\\
\bullet\ar@/_4ex/[rrr] &  &  & \bullet\ar[u] & \bullet\ar[d] 
&&& &&\bullet
\\
&&&&\bullet 
}
\]
The composite pictures are clearly not the same map in $\KL(\cU)$, and
so there is no guarantee that the two maps are the same in $\cC$.
They might be the same, but if so this is an ``accident''---it does
not follow from the basic axioms.

As one more example, let us consider the following composite:
\[\xymatrix{ S =SSS \ar[r] &  X^*X Y^*Y XX^* \ar[r] & X^* Y^* X X Y
X^* \ar[r]^{t_{X,X}} & X^* Y^* XX YX^*\ar[d] \\
&& S= SSS & X^* XY^* YXX^*\ar[l]
}
\]
Note that in the third map we have omitted the identity factors on
either side of the $t_{X,X}$, due to limitations of space.  All of the
other maps are the evident ones.  We claim that the composite can be
given a simpler description.  
Computing in $\KL(\{X,Y\})$ we get the following picture:

\[ \xymatrixcolsep{1.3pc}\xymatrixrowsep{0.9pc}\xymatrix{
{\scriptstyle{X}}\bullet\ar@/^3ex/[r] & \bullet\ar[dr] &
\bullet{\scriptstyle{Y}}\ar@/^3ex/[r]
& \bullet\ar[dr] & \bullet\ar[dl] & \bullet{\scriptstyle{X}}\ar@/_3ex/[l] \\
\bullet \ar[u] & \bullet\ar[ur] &\bullet\ar[dr] &
\bullet\ar[dl]  & \bullet\ar[d] &\bullet\ar[u] \\
\bullet\ar[u] & \bullet\ar[u] &\bullet\ar[dl] &
\bullet\ar[dr] & \bullet\ar[dl] &\bullet\ar[u] \\
\bullet\ar[u] & \bullet\ar@/^3ex/[l] &\bullet\ar[ul] &
\bullet\ar@/^3ex/[l] & \bullet\ar@/_3ex/[r] &\bullet\ar[u]
}
\]

\vspace{0.3in}

\noindent
This picture breaks up into two loops, one where the vertices are all
labelled by $X$ and the other where they are all labelled by $Y$.  As
a map from $S$ to $S$ in $\KL(\{X,Y\})$  this composite is therefore
equal to $\tr(\id_X)\circ \tr(\id_Y)$ (note that the order of
composition does not matter, since $\Hom(S,S)$ is commutative).  Since
this identity holds in the universal example $\KL(\{X,Y\})$, it also
holds in $\cC$.

\begin{remark}[Traces in Kelly-Laplaza categories]
Let $w$ be an object in $\KL(\{\cU\})$ and let $f\colon w\ra w$ be a
map.  Then $\tr(f)$ is a map $S\ra S$ in $\KL(\{\cU\})$.  We leave it
as an easy exercise to verify that $\tr(f)$ is represented by the
following picture:

\begin{picture}(300,80)
\put(150,0){\includegraphics[scale=0.25]{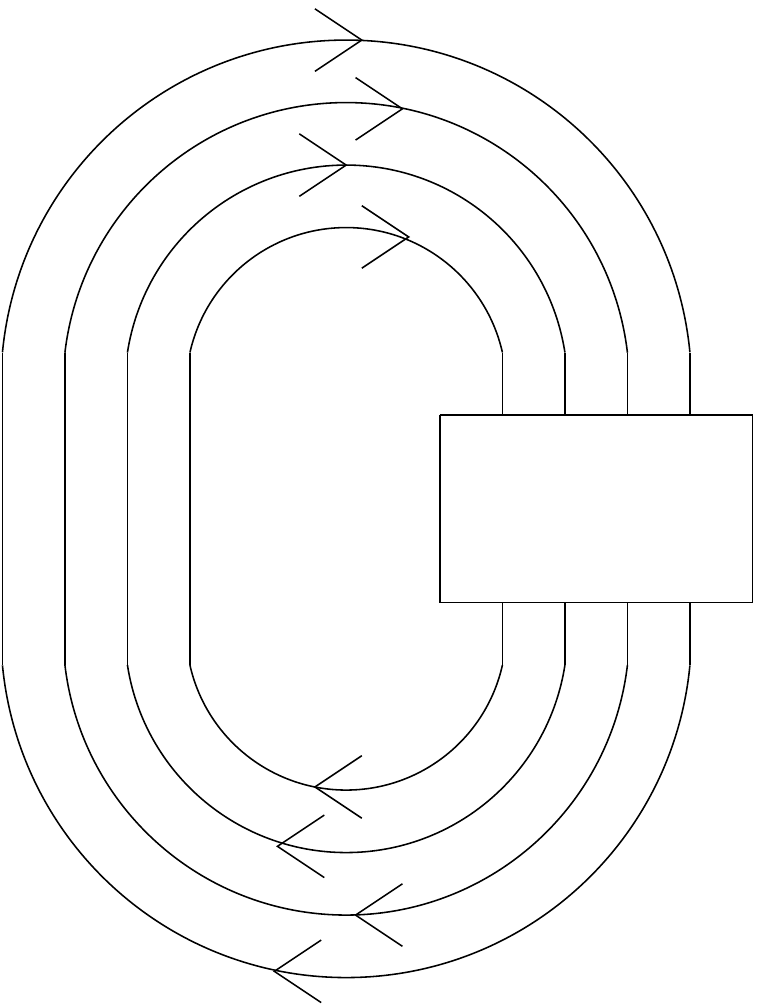}}
\put(190,33){$f$}
\end{picture}

\noindent
(where the picture representing $f$ should be inserted into the
blank box).  
\end{remark}

\section{Invertible objects}
\label{se:inv}
In this section we review the notion of invertible object in a
symmetric monoidal category, and establish some of their properties.  
For every invertible object $X$ we define a {\it basic commuter\/}
$\tau_X$, which is an isomorphism $\tau_X\colon S\ra S$ such that
$\tau_X^2=1$.  

\medskip

\subsection{Prelude}
Throughout this section $(\cC,\tens,S)$ is a symmetric monoidal
category.
If $u\colon S\ra S$ and $g\colon A\ra B$ then denote the composites
\[ A\iso S\tens A\llra{u\tens g} S\tens B \iso B \quad\text{and}\quad
A\iso A\tens S\llra{g\tens u} B\tens S\iso B
\]
by $u\ctens g$ and $g\ctens u$.  We will sometimes omit the carat and
just write $u\tens g$ and $g\tens u$ by abuse, but at other times it
is useful
to remember that $u\ctens g$ and $u\tens g$ are somewhat different.  

\begin{lemma}
\label{le:ctens}
Let $u\colon S\ra S$ and $g\colon A\ra B$.  Then 
\[ u\ctens g=(u\ctens
\id_B)\circ g=g\circ(u\ctens \id_A)=g\ctens u.
\]
\end{lemma}

\begin{proof}
This is elementary, using that $u\tens g=(\id_S\tens g)\circ
(u\tens\id_A)=(u\tens \id_B)\circ (\id_S\tens g)$ and similarly for
$g\tens u$.  
\end{proof}

\begin{remark}
\label{re:move-around}
The above lemma will often be used in the following way.  Suppose that
$g\colon A\ra B$ and $f\colon B\ra C$.  Then multiple applications of
the lemma give
\[ (u\ctens f)\circ g=f\circ (u\ctens \id_B)\circ g=f\circ(u\ctens g)=f\circ g\circ
(u\ctens \id_A)=u\ctens(fg).
\]
So we can move a $u\ctens(\blank)$ from anywhere inside a composite to
anywhere else, including outside the composite.  
\end{remark}

\begin{remark}
\label{re:endo}
Observe that for any object $V$ in $\cC$ we obtain a map of monoids
$\End(S)\ra \End(V)$ given by $u\mapsto u\ctens \id_V$.  
\end{remark}

\subsection{Invertible objects}
\begin{defn} An object $X$ in $\cC$ is \dfn{invertible} if there exists
  a $Y$ in $\cC$ and an isomorphism $\alpha\colon S\ra Y\tens X$.  
\end{defn}

It is easy to prove that if such a $Y$ exists then it is unique up to
isomorphism.  But to give an inverse for an object $X$, one must specify an object $Y$
together with the isomorphism $\alpha\colon Y\tens X \ra S$.  This map
$\alpha$ is not uniquely determined by $Y$, since one can clearly get
a different $\alpha$ by composing with an automorphism of $S$.  
Note that if $X$ and $Z$ are invertible then clearly so
is $X\tens Z$.  

We will often use the following observation:

\begin{prop}
\label{pr:inv-end}
If $X$ is an invertible object in $\cC$  then the canonical map $\End(S)\ra
\End(X)$ is an isomorphism of monoids.  
More generally, for any object $V$ of
$\cC$ the two maps $\End(V)\ra \End(V\tens X)$ and $\End(V)\ra \End(X\tens
V)$ (obtained by tensoring with identity maps) are both isomorphisms.
\end{prop}

\begin{proof}
Choose an inverse $(Y,\alpha)$ for $X$.
The functor $T_X\colon \cC\ra \cC$ given by $Z\mapsto Z\tens X$ is an
equivalence of categories, because an inverse is  given by
$W\mapsto W\tens Y$.  Since $T_X$ is an equivalence, for any $V$ in
$\cC$ the map $\End(V)\ra \End(T_X(V))$ is an isomorphism of monoids.
A similar argument shows $\End(V)\ra \End(X\tens V)$ to be an
isomorphism (or use the twist map $X\tens V\iso V\tens X$).  

When $V=S$ one has $T_X(S)\iso X$ via the unital isomorphism, and the
composite $\End(S)\ra \End(T_X(S))\iso \End(X)$ is readily checked to
be the map of Remark~\ref{re:endo}.
\end{proof}

When $X$ is invertible it will be useful to have a description of the
inverse to the isomorphism $\End(S)\ra \End(X)$.
If $(Y,\alpha)$ is a choice of inverse for $X$
and $g\colon X\ra X$, define $D_Y(g)$ to be the composite
\[ S \llra{\alpha} Y\tens X \llra{\id\tens g} Y\tens X
\llra{\alpha^{-1}} S.
\]
An easy diagram chase shows that $D_Y(u\ctens\id_X)=u$, for $u\in
\End(S)$.  Since $\End(S)\ra \End(X)$ is an isomorphism this verifies
that $D_Y$ is the inverse, and so in particular does not depend on the
choice of $(Y,\alpha)$.  From now on we will just write $D(g)$ rather
than $D_Y(g)$.  

The homomorphism $D\colon \End(X)\ra \End(S)$ is a bit like a trace,
but it does not coincide with the standard trace that exists for
dualizable objects as defined in  Section~\ref{se:KL-pre} (see
Remark~\ref{re:tr-D} for an explicit example).  The map $D$ can also be regarded
as something like a determinant---this analogy works well when $\cC$ is a category
of vector spaces or vector bundles, but of course in those cases the
determinant and trace are indistinguishable on one-dimensional
objects.  In the present paper we will just call $D(f)$ the
``$D$-invariant'' of the map $f\colon X\ra X$.  Like a determinant,
the $D$-invariant is multiplicative (being a homomorphism of monoids):
that is, 
$D(\id_X)=\id_S$ and $D(fg)=D(f)D(g)$.
Here are some further properties of the $D$-invariant:

\begin{lemma} Let $X$ and $Z$ be invertible objects.
\label{le:trace}
\begin{enumerate}[(a)]
\item Given a commutative diagram
\[ \xymatrix{
X \ar[d]_f\ar[r]^q & Z\ar[d]^h \\
X \ar[r]^q & Z
}
\] 
in which $q$ is an
isomorphism, 
one has $D(f)=D(h)$.  
\item If $f\colon X\ra X$, then $D(\id_Z\tens f)=D(f\tens
  \id_Z)=D(f)$.  
\item If $f\colon X\ra X$ and $g\colon Z\ra Z$ then $D(f\tens
  g)=D(f)D(g)$, where the product on the right-hand-side 
is in the monoid $\End(S)$.  
\end{enumerate}
\end{lemma}

\begin{proof}
All of the parts are easy exercises.  For (a) one uses the diagram
\[ \xymatrix{
\End(S) \ar[r]\ar[dr] & \End(X)\ar[d]\\
&\End(Z)
}
\]
where the vertical arrow sends a map $f$ to $qfq^{-1}$.  One checks
that the diagram
commutes using Remark~\ref{re:move-around}, and then
it follows at once that $D(qfq^{-1})=D(f)$.

For (b) one looks at the
composite
$\End(S) \ra \End(X) \ra \End(X\tens Z)$.
Both maps are isomorphisms, $D_X$ is the inverse of the first map, and
$D_{X\tens Z}$ is the inverse of the composite; it follows at once
that $D_{X\tens Z}(f\tens \id_Z)=D(f)$.  

Finally, (c) follows from (b) and the fact
that $f\tens g=(f\tens \id_Z)\circ (\id_X\tens g)$.  
\end{proof}

\begin{remark}
\label{re:move-around2}
Let $f\colon A\ra B$ and $g\colon B\ra B$, where $B$ is invertible.
The $D$-invariant of $g$ is the unique map $S\ra S$ satisfying
$g=D(g)\ctens \id_B$.  We can then write
\[ gf=(D(g)\ctens \id_B)\circ f=f\circ (D(g)\ctens \id_A),
\]
using Remark~\ref{re:move-around} for the second equality.  So
automorphisms of invertible objects can effectively be moved around
inside a composition, by replacing them with their $D$-invariant.   
\end{remark}

\subsection{The adjoint to $\alpha$ and the trace of a map}

The following result shows that an invertible object is left
dualizable.  

\begin{prop}
\label{pr:alpha-hat}
Let $X$ be an invertible object in $\cC$, with inverse $(Y,\alpha)$.
Then there is a unique map $\hat{\alpha}\colon X\tens Y \ra S$ with the
property that the composite
\[ X\iso X\tens S \llra{\id\tens \alpha} X\tens (Y\tens X) \iso (X\tens
Y)\tens X \llra{\hat{\alpha}\tens \id} S\tens X \iso X\]
equals the identity.
Moreover, $\hat{\alpha}$ is an isomorphism and the composite
\[
Y\iso S\tens Y \llra{\alpha\tens \id} (Y\tens X)\tens Y \iso Y\tens
(X\tens Y) \llra{\id\tens \hat{\alpha}} Y\tens S\iso Y
\]
also equals the identity.
\end{prop}

\begin{remark}
\label{re:GrVect}
Note that one is tempted to assume that $\hat{\alpha}$ equals the composite
\begin{myequation}
\label{eq:comp}  X\tens Y \llra{t} Y\tens X \llra{\alpha^{-1}} S.
\end{myequation}
This need not be the case.  Let $k$ be a field and let
$\cC=\GrVect_k^{\pm}$ be the category of $\Z$-graded vector spaces
with the usual tensor product, and with the twist map that involves
signs. Let $X=k[1]$, and
$Y=k[-1]$.  Let $\alpha\colon k \ra Y\tens X$ send $1$ to $1\tens 1$.
Then $\hat{\alpha}\colon X\tens Y \ra k$ must be the multiplication
map, whereas the composite (\ref{eq:comp}) sends $a\tens b$ to $-ab$.
\end{remark}

\begin{proof}[Proof of Proposition~\ref{pr:alpha-hat}]
Let $F_X\colon \cC \ra \cC$ be given by $F_X(A)=A\tens X$, and let
$F_Y\colon \cC \ra \cC$ be given by $F_Y(A)=A\tens Y$.  These are an
equivalence of categories.  Consider the chain of functions
\[ \cC(X\tens Y,S) \llra{F_X} \cC((X\tens Y)\tens X,S\tens X)
\llra{\gamma_*} \cC(X,S\tens X) \llra{\iso} \cC(X,X)\]
where $\gamma\colon X\ra (X\tens Y)\tens X$ is the composite
\[
X \llra{\iso} X\tens S \llra{\id\tens \alpha} X\tens (Y\tens X)
\llra{\iso} (X\tens Y)\tens X.
\]
Since $\gamma$ is an isomorphism, $\gamma_*$ is a bijection.
Likewise, the map labelled $F_X$ is a bijection because $F_X$ is
part of an equivalence of categories.  So the whole composite is a
bijection, and we let $\hat{\alpha}$ be the unique map whose image is the
identity $\id_X$.  This justifies the first claim in the statement of
the proposition.

To see that $\hat{\alpha}$ is an isomorphism, consider again the first
composite in
the statement.  All the maps other than $\hat{\alpha}\tens
\id$ are known to be isomorphisms, so we can conclude the same for
$\hat{\alpha}\tens \id=F_X(\hat{\alpha})$.  As $F_X$ is part of an equivalence of
categories, it must be that $\hat{\alpha}$ was an isomorphism itself.

Finally, let $f\colon Y\ra Y$ be the second composite in the
statement.  We have a commutative diagram
\[ \xymatrixcolsep{4pc}\xymatrix{
X\tens Y \ar[dr]_{\id}\ar[r]^-{\id\tens \alpha\tens \id} & X\tens Y\tens X\tens Y
\ar[d]^{\hat{\alpha}\tens \id\tens \id}\ar[r]^-{\id\tens \id\tens \hat{\alpha}} & X\tens
Y\tens S \ar[r]^\iso & X\tens Y \ar[d]^{\hat{\alpha}} \\
&X\tens Y \ar[rr]^{\hat{\alpha}} && S,
}
\]
where we have left off several associativity isomorphisms.  The
triangle commutes by the defining property of $\hat{\alpha}$, and 
the composite across the top row is $\id_X\tens f$.  
So the diagram shows that $\hat{\alpha}\circ(\id_X\tens f)=\hat{\alpha}$.
As $\hat{\alpha}$ is an isomorphism we conclude that $\id_X\tens
f=\id_{X\tens Y}$.  From this it follows that $f=\id_Y$, by
Proposition~\ref{pr:inv-end}.  
\end{proof}

Suppose that $X$ is invertible and $f\colon X\ra X$.  Since $X$ is
left dualizable we may take the trace, obtaining $\tr(f)\colon S\ra
S$.  Recall that we have another way to obtain a self-map of $S$,
namely the $D$-invariant $D(f)$.  
These are connected by the
following formula:

\begin{prop}
\label{pr:D-tr}
Let $X$ be an invertible object and let $f\colon X\ra X$.  Then
one has $\tr(\id_X)\cdot D(f)=\tr(f)$.  
\end{prop}

\begin{proof}
The composite $\tr(\id_X)\cdot D(f)$ is
\[ S\llra{\alpha} Y\tens X \llra{\id\tens f}Y\tens X \llra{\alpha^{-1}} S
\llra{\alpha}Y\tens X
\llra{\id} Y\tens X \llra{t} X\tens Y \llra{\hat{\alpha}} S.
\]
The terms in the middle cancel and we obtain the definition of
$\tr(f)$.  
\end{proof}

\begin{remark}
\label{re:tr-D}
Consider again the example $\GrVect^{\pm}_k$ from
Remark~\ref{re:GrVect}, with $X$, $Y$, and $\alpha$ as described there.
Then $\tr(\id_X)=-1$, but of course $D(\id_X)=1$.  So this gives an
example where the trace and $D$-invariant are distinct.
\end{remark}

Our next major goal will be to prove that when $X$ is invertible one
has $\tr(\id_X)^2=\id_S$.  This is an important property of invertible
objects, but unfortunately we have not been able to find a direct, simple-minded
proof.  We will instead deduce the result from a cyclic
permutation property, which {\it a priori\/} feels somewhat deeper.

\subsection{Automorphisms of invertible objects induced from permutations}

Let $x_1,\ldots,x_n$ be formal variables and let $w$ be any tensor
word in the $x_i$'s with the property that each $x_i$ appears exactly
once.  For instance, if $n=3$ we might have $w=(x_1\tens x_3) \tens
x_2$.  We can associate to $w$ a functor $F_w\colon \cC^n \ra \cC$
which plugs in objects for the variables $x_i$. 
For objects $X_1,\ldots,X_n$ in $\cC$ write
$w(X_1,\ldots,X_n)$ as shorthand for $F_w(X_1,\ldots,X_n)$, and write
$w(X)$ as shorthand for $F_w(X,X,\ldots,X)$.

If $\sigma$ is a permutation of $\{1,\ldots,n\}$, we let $w\sigma$
denote the word in which $x_i$ has been replaced by $x_{\sigma(i)}$.  So we
can write 
\[ F_{w\sigma}(X_1,\ldots,X_n)=F_w(\sigma\cdot
(X_1,\ldots,X_n))=F_w(X_{\sigma(1)},X_{\sigma(2)},\ldots,X_{\sigma(n)}).
\]  

By MacLane's coherence theorem
\cite[Theorem XI.1.1]{M} there is a unique natural
transformation $F_w \ra F_{w\sigma}$ obtained by composing
associativity and commutativity isomorphisms.  If $X$ is an object in
$\cC$, we can evaluate this
natural transformation at the  tuple $(X,X,\ldots,X)$ and thereby obtain an
automorphism $\phi_{w,\sigma}\colon w(X)\ra w(X)$.  In this way we
obtain a function $\phi_w\colon \Sigma_n \ra \Aut(w(X))$, which is readily checked
to be a homomorphism.

The following result is from \cite[Discussion preceding Theorem 4.3]{V}:

\begin{lemma}
\label{le:cyclic}
Let $X$ be an invertible object in $\cC$.  Then for any tensor word
$w$ in $n$ variables, and any even permutation $\sigma$ in $\Sigma_n$,
the map $\phi_{w,\sigma}\colon w(X) \ra w(X)$ is equal to the
identity.
In particular, the composite map
\[ (X\tens X) \tens X \llra{t_{X\tens X,X}} X\tens (X\tens X)
\llra{a} (X\tens X)\tens X \]
is equal to the identity.  (Note that this composite map is an
instance of the canonical map $(A\tens B)\tens C \ra C\tens (A\tens
B)$, i.e. the cyclic permutation map).
\end{lemma}

\begin{proof}
Since $X$ is invertible, so is $w(X)$.  Hence
$\Aut(w(X))$ is abelian by Proposition~\ref{pr:inv-end} and
Lemma~\ref{le:aut-ab}.  
This means the homomorphism $\phi_w\colon
\Sigma_n \ra \Aut(w(X))$ kills the commutator subgroup of $\Sigma_n$,
which is the alternating group $A_n$.
\end{proof}

We can now obtain our goal:

\begin{prop}
\label{pr:tr-inv}
If $X$ is an invertible object then $\tr(\id_X)^2=\id_S$.  
\end{prop}

\begin{proof}
Let $f=\tr(\id_{(X\tens X)\tens X})$.  In the Kelly-Laplaza category
$\KL(\{X\})$ this is represented by the picture

\begin{picture}(300,60)
\put(150,0){\includegraphics[scale=0.25]{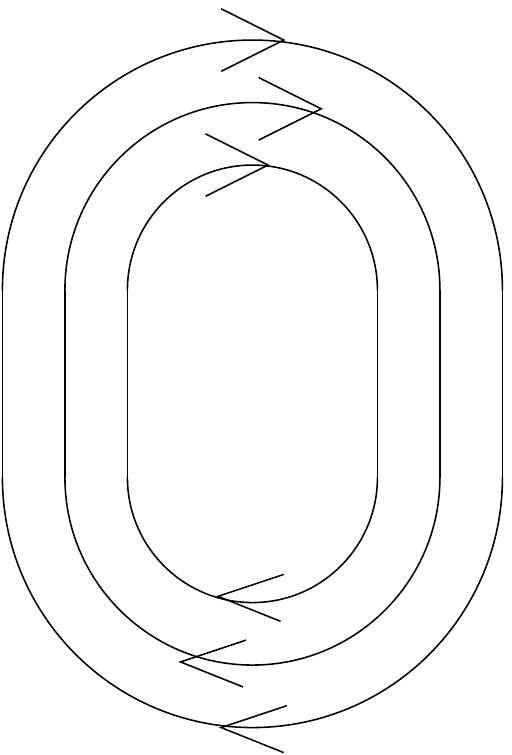}}
\end{picture}

\noindent
from which one obtains that $f=\tr(\id_X)^3$ (due to the three components).  

On the other hand, by Lemma~\ref{le:cyclic} $f$ is also equal to
$\tr(c)$ where $c=a\circ t_{X\tens X,X}$ is the cyclic permutation
map.  In the Kelly-Laplaza category this is represented by the picture

\begin{picture}(300,60)
\put(150,0){\includegraphics[scale=0.25]{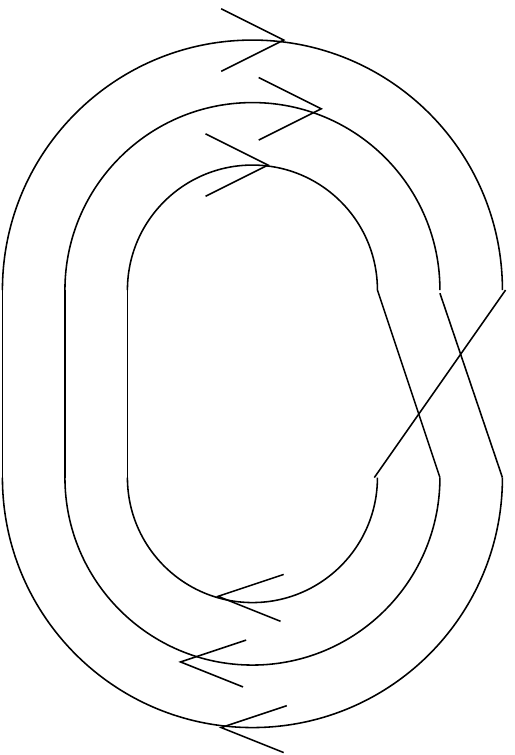}}
\end{picture}

\noindent
and from this one obtains $f=\tr(\id_X)$ (due to the one connected
component).  So we have proven that $\tr(\id_X)^3=\tr(\id_X)$.  

However, when $X$ is invertible all the maps in the composite defining
$\tr(\id_X)$ are isomorphisms---so
$\tr(\id_X)$ is an isomorphism.  We can therefore cancel one $\tr(\id_X)$
from each side of the previous equation to obtain
$\tr(\id_X)^2=\id_S$.  
\end{proof}

\subsection{Basic commuters for invertible objects}
If $X$ is an invertible object define $\tau_X=\tr(\id_X)\in \Aut(S)$
and call this the \dfn{basic commuter} associated to $X$.  These
elements will be important in our treatment of skew-commutativity in
Section~\ref{se:app}.
Note that
if $X\iso X'$ then $\tau_X=\tau_{X'}$, by
Proposition~\ref{pr:trace}(a).
Let $\Pic(\cC)$ denote the set of isomorphism classes of invertible
objects in $\cC$; the tensor product makes this set into a group.  We
have produced a set map $\tau\colon \Pic(\cC)\ra \Aut(S)$, whose image
lands in the $2$-torsion subgroup $\psub{2}{\Aut(S)}$ by
Proposition~\ref{pr:tr-inv}.

\begin{prop}
\label{pr:tau}  The map $\tau\colon \Pic(\cC)\ra \psub{2}{\Aut(S)}$ is
a group homomorphism.  
\end{prop}

\begin{proof}
We need only show that if $X$ and $Z$ are invertible then
$\tau_{X\tens Z}=\tau_X\cdot \tau_Z$.  For this
we work in 
the Kelly-Laplaza category
$\KL(\{X,Z\})$ and observe that the map $\tr(\id_{X\tens Z})$ is represented by the
picture

\begin{picture}(300,50)
\put(150,0){\includegraphics[scale=0.35]{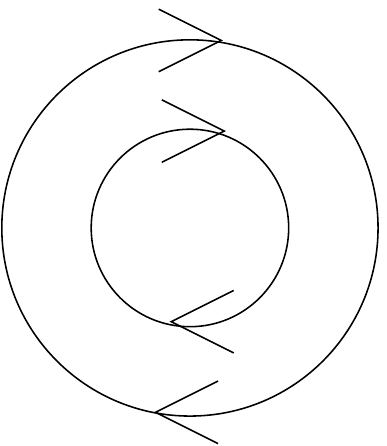}}
\put(172,19){$\scriptstyle{X}$}
\put(190,19){$\scriptstyle{Z}$}
\end{picture}

\noindent
which also represents $\tr(\id_X)\cdot \tr(\id_Z)$. 
\end{proof}

We may also describe $\tau_X$ in terms of the twist map 
$t_{X,X}\colon X\tens X\ra X\tens X$:

\begin{prop}
\label{pr:tr=D}
$\tau_X=\tr(\id_X)=D(t_{X,X})=\tr(t_{X,X})$.
\end{prop}

\begin{proof}
The first equality is the definition.  By
Proposition~\ref{pr:D-tr} we have that
$\tr(\id_{X\tens X})\cdot D(t_{X,X})=\tr(t_{X,X})$,  but
$\tr(\id_{X\tens X})=\tau_{X\tens X}=\tau_X^2=1$ using
Proposition~\ref{pr:tau}.  This proves the third equality.  We
can complete the proof by showing that $\tr(\id_X)=\tr(t_{X,X})$.
This actually follows  by the Kelly-Laplaza theorem, for
$\tr(t_{X,X})$ is represented by the picture below:

\begin{picture}(300,50)
\put(150,0){\includegraphics[scale=0.25]{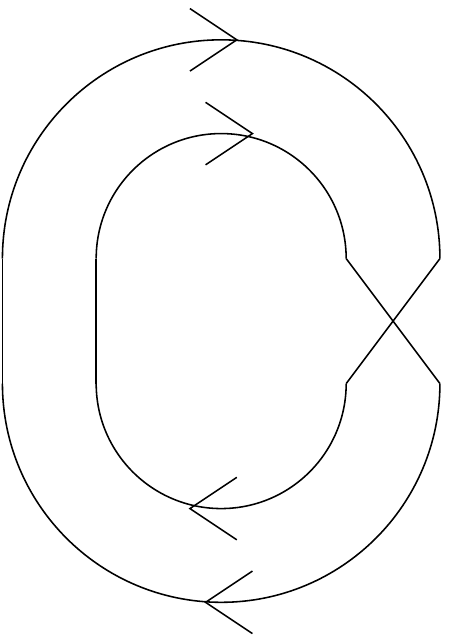}}
\end{picture}

\noindent
Because there is only one connected component, this also represents
$\tr(\id_X)$.  
\end{proof}


\section{Coherence for invertible objects}
\label{se:coherence}
In this section we prove our main coherence theorems for invertible
objects in a symmetric monoidal category.  We deduce these as
consequences of the Kelly-Laplaza theorem.  

\medskip

Let $(\cC,\tens,S)$ be a symmetric monoidal category and let $X$ be an
invertible object with inverse $(X^*,\alpha)$.  Recall
that $\tr(\id_X)$ is defined to be the composite
\[ S \llra{\alpha} X^*\tens X \llra{t} X\tens X^* \llra{\hat{\alpha}}
S.
\]
Every map in this composite is an isomorphism, so we can write
\[ \alpha^{-1} = \tr(\id_X)^{-1}\circ (\hat{\alpha} t) =
\tr(\id_X)\circ (\hat{\alpha} t) = \tr(\id_X)\ctens (\hat{\alpha}t)\]
where in the second equality we have used that $\tr(\id_X)^2=\id_S$
(see Proposition~\ref{pr:tr-inv}).  
Similarly, we have
\[ \hat{\alpha}^{-1}=\tr(\id_X)\ctens (t\alpha).\]

Let $\cC^{inv}$ be the full subcategory of invertible objects.  Then
$\cC^{inv}$ is a symmetric monoidal category with left duals.  We will
deduce our desired coherence theorems for invertible objects from 
Kelly-Laplaza coherence applied to $\cC^{inv}$.  

We saw in Section~\ref{se:KL} that maps in a Kelly-Laplaza category
can be represented by pictures consisting of certain kinds of directed curves in the
plane.  These curves are very simple: every crossing is a standard
``X''-crossing, and every place where there is a horizontal tangent
line is either a local minimum or local maximum with respect to the
$y$-coordinate (i.e., a cup or cap)---we will call these cups and caps
the \dfn{critical
points} of the curve.

\subsection{Coherence without self-twists}

\begin{proof}[Proof of Theorem~\ref{th:coherence}]
Let us use the term ``acceptable'' for 
 formal composites of the type considered in the statement
of the proposition.
Suppose there were two acceptable formal composites $f,g\colon w_1\ra
w_2$ that yielded
different maps in $\cC$.  Then the formal composite $g^{-1}f\colon
w_1\ra w_1$ would yield a map in $\cC$ different from the identity.
So it suffices to prove the proposition in the case $w_1=w_2$ and $g=\id$.  

Suppose $f\colon w_1\ra w_1$ is an acceptable formal composite such that
$f(\cC)\neq \id$. 
Let $w_1^*$ be the formal inverse of the word $w_1$, and choose any
acceptable formal composite $h\colon S\ra w_1^*\tens w_1$.  Consider the
formal composite
\[ S \llra{h} w_1^*\tens w_1 \llra{\id \tens f} w_1^*\tens w_1
\llra{h^{-1}} S.
\]
Since
$f(\cC)\neq \id$ Proposition~\ref{pr:inv-end} shows that 
the map $\id\tens f$ also does not give the identity in $\cC$, and
from this it follows that the above composite
does not give the identity either.
So it suffices to prove the proposition in the case $w_1=w_2=S$ and $g=\id$.  

Let $n_1$ be the number of $\alpha^{-1}$ maps that appear in $f$, and
let $n_2$ be the number of $\hat{\alpha}^{-1}$ maps that appear.  Let
$n=n_1+n_2$.  
Let $F$ be the formal composite in which every $\alpha^{-1}$ has been
replaced with $\hat{\alpha}t$ and every $\hat{\alpha}^{-1}$ has been
replaced with $t\alpha$.  
Using Remark~\ref{re:move-around}, the identities 
$\alpha^{-1}=\tr(\id_X)\ctens (\hat{\alpha}t)$ and
$\hat{\alpha}^{-1}=\tr(\id_X)\ctens (t\alpha)$ show that 
$f(\cC)=
\tr(\id_X)^n\ctens F(\cC)$.

The reason for introducing $F$ is that it only involves maps that
exist for dualizable objects, rather than invertible ones.  So we may
consider $F$ as a composite in the
category Kelly-Laplaza category $\KL(\{X\})$.   The assumption that $f$ was
acceptable implies that $F$ can be represented by the disjoint union
of simple closed curves; for example, one of the components might look
like this:

\begin{picture}(300,120)
\put(100,0){\includegraphics[scale=0.25]{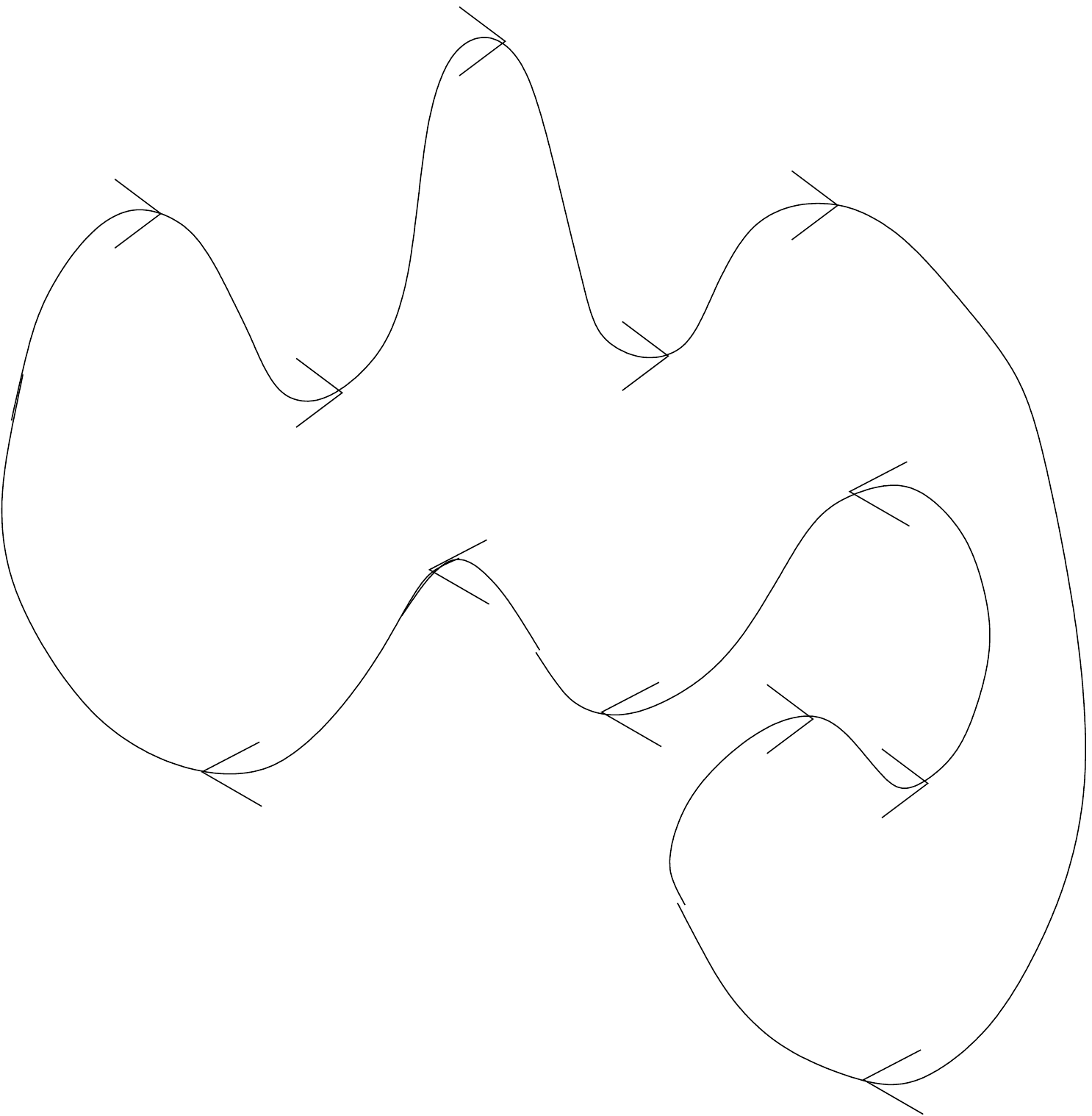}}
\end{picture}

\noindent
Let us be clear about why this works.  The assumption that the formal
composite $f$ is acceptable guarantees that the only twist maps that
appear in $F$ come together with an $\alpha$ or $\hat{\alpha}$.  In
terms of the pictures, each of these twists can be eliminated; to see
this, recall
how the pictures work:

\vspace{0.2in}
  
\[ \xymatrix{\bullet \ar@/^3ex/[r]&\bullet}=\alpha,\qquad
\xymatrix{\bullet \ar@/_3ex/[r]&\bullet}=\hat{\alpha},
\]
\vspace{0.2in}


\[
\xymatrixrowsep{0.5pc}\xymatrixcolsep{1.3pc}
\xymatrix{ &&& \bullet\ar@/^3ex/[r] & \bullet\ar[ddl]\\
&&=&&&=t\alpha,\\
\bullet &\bullet
\ar@/_3ex/[l] & 
& \bullet & \bullet\ar[uul]}
\qquad
\xymatrixrowsep{0.5pc}\xymatrixcolsep{1.3pc}
\xymatrix{ &&& \bullet&\bullet\ar[ddl] & \\ 
&&=&&&=\hat{\alpha}t.\\
\bullet &\bullet
\ar@/^3ex/[l]
&& \bullet\ar@/_3ex/[r]&\bullet\ar[uul]}
\]

\vspace{0.2in}
\noindent
The twists in $F$ only appear in conjunction with a cup or cap, and so
they can all be depicted by an untwisted cup or cap going in the
opposite direction.

Now we come to the crux of the matter.  For the moment assume that the
picture for $F$ only contains one component, for simplicity.  In this
simple closed curve, every $\alpha$ or $\hat{\alpha}$ appears as a
right-pointing arrow on a cup or cap.  Likewise, every $t\alpha$ or
$\hat{\alpha}t$ appears as a left-pointing arrow on a cup or cap.  So
the number of left-pointing arrows in our simple closed curve is
$n_1+n_2=n$.  But elementary topology shows that in a (nice enough)
directed, simple, closed curve the number of left-pointing critical
points  must always be odd (the same is true for the
number of right-pointing critical points, of course).  So $n$ is odd.

In the category $KL(\{(X\})$ we know that a simple closed curve as
above is equal to $\tr(\id_X)$.  So when $F$ is evaluated in $\cC$
it also gives this trace.  Putting everything together, we find that
$f(\cC)=\tr(\id_X)^n\circ F(\cC)=\tr(\id_X)^n\circ
\tr(\id_X)=\tr(\id_X)^{n+1}$.  
Given that $n$ is odd, this just equals the identity (using
Proposition~\ref{pr:tr-inv}).  This completes the proof for the case
that the picture for $F$ has only one component.

For the multi-component case we observe that in $\KL(\{X\})$ the
map $F$ is the composition of the maps represented by each individual
component; by what has already been argued, $F(\cC)$ is a
composition of identity maps and hence equal to the identity.
\end{proof}

\begin{proof}[Proof of Theorem~\ref{th:multi-coherence1}]
This is essentially the same as the proof of
Theorem~\ref{th:coherence}, except we use the Kelly-Laplaza
category $\KL(\{X_1,\ldots,X_n\})$.  Note that the existence part of
the theorem is obvious; the work lies in showing uniqueness of the
isomorphism.  For this one reduces, just as in
Theorem~\ref{th:coherence}, to the case of a composite $f$ that starts
and ends with $S$ and is of the type specified in the statement of the
theorem.

The composite $f$ is
then replaced by a corresponding formal composite $F$ in
$\KL(\{X_1,\ldots,X_n\})$.  The picture for $F$ is a collection of
simple, closed curves in the plane, each labelled by one of the
$X_i$'s, which are allowed to intersect each other in double points.
The Kelly-Laplaza theorem identifies $F$ with the composite of the
maps whose pictures correspond to each closed curve.  In this way one
reduces to the one-variable case handled by
Theorem~\ref{th:coherence}, to conclude that $f$ must be the identity.   
\end{proof}

\subsection{Coherence with self-twists}

\begin{proof}[Proof of Theorem~\ref{th:coherence2}]
The proof proceeds along the same lines as that of
Theorem~\ref{th:coherence}.  One immediately reduces to the case
$w_1=w_2=S$, $g=\id$, and where the parity of $f$ is even.
Just as before, we replace $f$ by a corresponding formal
composite $F$ in the Kelly-Laplaza category $\KL(\{X\})$.  We have
\[ f(\cC)=\tr(\id_X)^{n}\circ F(\cC),
\]
where the integer $n$ is the same as in the proof of
Theorem~\ref{th:coherence}.  

The picture corresponding to $F$ is no longer a union of simple curves
as it was in the proof of Theorem~\ref{th:coherence}.  Rather, it
is a union of oriented, closed curves that may contain double points of
self-intersection.  For pedagogical purposes let us first deal with
the case where the picture contains a single closed curve, for example
as follows:

\begin{picture}(300,120)
\put(100,10){\includegraphics[scale=0.2]{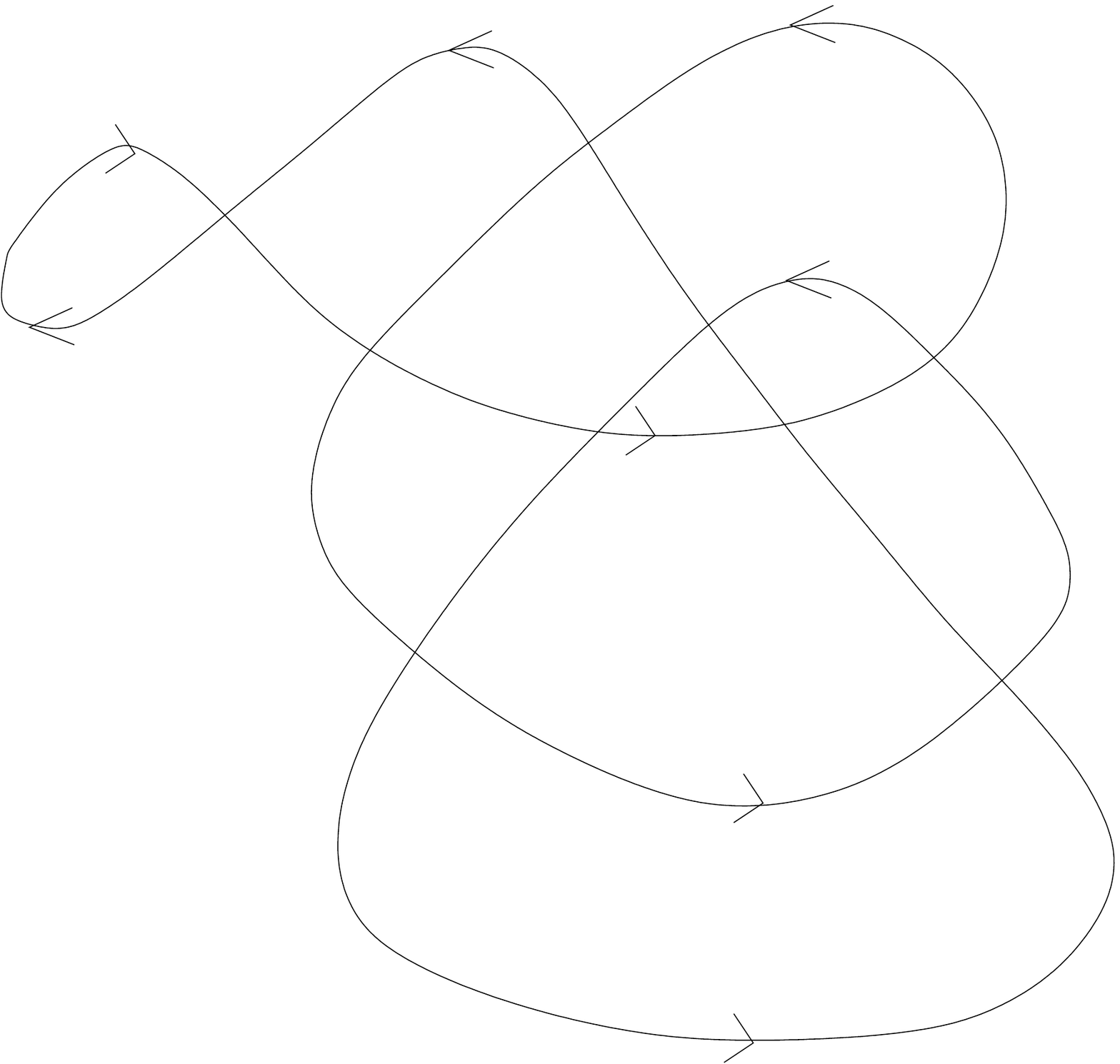}}
\end{picture}

It is easy to prove that in such an oriented curve one has
\begin{myequation}
\label{eq:critical}
\qquad\qquad (\text{\# of left-pointing critical points}) + (\text{\# of double points})
\equiv 1\ \text{mod 2}.
\end{myequation}
Indeed, let $L$ denote the number on the left of the congruence.
Imagine taking a closed loop of string---an unknot---and laying it on
top of the plane containing our oriented curve, in such a way that the string
exactly covers the curve.  This gives us an oriented knot diagram
which is similar to our original picture but in which
every double-point has been changed to an over- or under-crossing.  
One readily checks that the parity of the number $L$ is unchanged
under the Reidemeister moves.  Since our knot diagram is equivalent to
the unknot, this says that the parity of $L$ is the same as the
corresponding number for an oriented circle.  But a circle clearly
gives $L=1$.  

The number of left-pointing critical points in the picture for $F$ is
just the number $n$.  Likewise, the number of double points in the
picture is the parity of the formal composite $f$, which we
have assumed to be even.  So (\ref{eq:critical}) tells us that $n$ is odd.

The Kelly-Laplaza theorem implies that $F(\cC)=\tr(\id_X)$.  So
$f(\cC)=\tr(\id_X)^{n+1}$.  Since $n+1$ is even
$\tr(\id_X)^{n+1}=\id_S$ by Proposition~\ref{pr:tr-inv}, and this
completes the proof in the present case (where the picture for $F$
contains one closed curve).

For the
general case we have $F(\cC)=\tr(\id_X)^e$ where $e$ is the number of
closed curves in the picture.  The analog of
(\ref{eq:critical})---whose proof is the same as before---becomes
\begin{myequation}
\label{eq:critical2}
\qquad (\text{\# of left-pointing critical points}) + (\text{\# of
double points}) \equiv e \text{\ mod 2}.
\end{myequation}
We then obtain $f(\cC)=\tr(\id_X)^{n+e}$, and (\ref{eq:critical2})
yields that $n+e$ is even.  So again we have $f(\cC)=\id_S$, as desired.
\end{proof}

\begin{proof}[Proof of Theorem~\ref{th:multi-coherence2}]
This is a straightforward generalization of the proof of
Theorem~\ref{th:coherence2}, in the same way that
Theorem~\ref{th:multi-coherence1} generalized
Theorem~\ref{th:coherence}.
The main point is that a map from $S$ to $S$ in
$\KL(\{X_1,\ldots,X_n\})$ is represented by a collection of
closed curves, each of which is labelled by one of the $X_i$'s.  The
Kelly-Laplaza theorem identifies such a map with the composite of the
maps obtained by considering each closed curve separately.  The
$i$-parity of our formal composite represents the number of double
points in the curves labelled by $X_i$.  The hypothesis that each of
these parities is even guarantees, just as in the proof of
Theorem~\ref{th:multi-coherence1}, that the specified map in
$\KL(\{X_1,\ldots,X_n\})$ is the identity.
\end{proof}


\section{The main applications: $\Z^n$-graded rings of maps}
\label{se:app}
Assume that $(\cC,\tens,S)$ is an additive category with a symmetric
monoidal structure, where the tensor product is an additive functor in
each variable.  In this section we investigate $\Z^n$-graded groups of
maps in $\cC$.

\medskip

Suppose $X_1,\ldots,X_n\in \cC$ are invertible objects, and let
$(X_i^*,\alpha_i)$ be a specific choice of inverse for $X_i$.  Recall
the definition of $\power{X}{a}$ for every $a\in \Z^n$, from
Section~\ref{se:intro}.
For every $a,b\in \Z^n$ there is a canonical
isomorphism
\[ \phi_{a,b}\colon \power{X}{a}\tens \power{X}{b} \llra{\iso}
\power{X}{a+b}
\]     
specfied by Theorem~\ref{th:multi-coherence1}.  The uniqueness
part of that proposition guarantees that
the pentagonal diagrams (\ref{eq:assoc}) all commute, and
that for $a,b\in \N^n$ these $\phi_{a,b}$'s coincide
with the ones defined in Section~\ref{se:MacLane}.

For any $W\in \cC$ define $\pi_*(W)$ to be the $\Z^n$-graded abelian
group 
\[ \pi_*(W) = \oplus_{a\in \Z^n} \cC(\power{X}{a},W).
\]
Suppose that $U$, $V$, and $W$ are objects and that there is a pairing
$U\tens V\ra W$.  The maps $\phi_{a,b}$ allow us to define a
$\Z^n$-graded pairing $\pi_*(U)\tens \pi_*(V)\ra \pi_*(W)$ as follows.
Suppose $f\colon \power{X}{a}\ra U$ and $g\colon \power{X}{b}\ra V$.
Define the product $f\cdot g$ to be the composite
\[ \power{X}{a+b}\llra{\phi_{a,b}^{-1}} \power{X}{a}\tens \power{X}{b}
\llra{f\tens g} U\tens V \lra W.
\]

\begin{prop}  
\label{pr:Z-ring}
Let $U$ be a monoid with respect to $\tens$.  
\begin{enumerate}[(a)]
\item $\pi_*(U)$ is a $\Z^n$-graded ring (associative and unital).
\item If $V$ is a left (resp. right) module over $U$ then $\pi_*(V)$
is a left (resp. right) module over $\pi_*(U)$.  
\item If $U$ is a commutative monoid then $\pi_0(U)$ is central in $\pi_*(U)$.
\end{enumerate}
\end{prop}

\begin{proof}
The proofs of (a) and (b) are the same: distributivity is
automatic, and associativity follows from the commutativity of the
diagram (\ref{eq:assoc}) involving the $\phi_{a,b}$'s.  The unit
conditions follow as in the proof of  Proposition~\ref{pr:R}.

For (c), let $f\colon \power{X}{a}\ra U$ and $g\colon S\ra U$.  The
following diagram is commutative:
\[ \xymatrix{
\power{X}{a}\ar[r]\ar[dr] & \power{X}{a}\tens S \ar[d]_t\ar[r]^{f\tens
g} & U\tens U \ar[d]_t\ar[r]_\mu & U \\
& S\tens \power{X}{a} \ar[r]^{g\tens f} & U\tens U \ar[ur]_\mu
}
\]
The composite across the top is $f\cdot g$, and the composite across the
bottom is $g\cdot f$.  Commutativity of the diagram shows these are equal.
\end{proof}

\subsection{Representation of elements in $\pi_*(S)$  by maps in $\cC$}
Let $w_1$ and $w_2$ be two tensor words in the symbols $X_i^{\pm 1}$,
and suppose that $f\colon w_1\ra w_2$ is a map.
Theorem~\ref{th:multi-coherence1} gives canonical isomorphisms
$\power{X}{a} \ra w_1$ and $\power{X}{b}\ra w_2$ for unique $a,b\in
\Z^n$.  From now on we will denote all canonical isomorphisms provided
by Theorem~\ref{th:multi-coherence1} by $\phi$.  (A consequence of
this is that a canonical map and its inverse---which is also
canonical---are sometimes both denoted by $\phi$; in practice this does
not lead to much confusion, though.)  Let $\ab{f}$
denote the composite 
\[ \power{X}{a} \llra{\phi} w_1 \llra{f} w_2 \llra{\phi} \power{X}{b}.
\]
There are two evident ways to obtain an element of $\pi_*(S)$ from
$\ab{f}$.    Let $[f]_r\in \pi_{a-b}(S)$ be the composite
\[ \power{X}{a-b} \llra{\phi} \power{X}{-b}\tens \power{X}{a}
\llra{\id \tens \ab{f}} \power{X}{-b}\tens \power{X}{b} \llra{\phi} S.
\]
and let $[f]_l\in \pi_{a-b}(S)$ be the composite
\[ \power{X}{a-b} \llra{\phi} \power{X}{a}\tens \power{X}{-b}
\llra{\ab{f}\tens \id} \power{X}{b}\tens \power{X}{-b} \llra{\phi} S.
\]
In general one must be careful, as $[f]_r$ and $[f]_l$ need not be the
same element.  We will give a precise formula for relating them in
Proposition~\ref{pr:l-r-general} below, but it
will take some work to build up to this.  We start with some simple
observations:

\begin{prop}
\label{pr:new}
Let $w_1$, $w_2$, and $w_3$ be three tensor words that are formally
isomorphic to  $\power{X}{a}$, $\power{X}{b}$, and $\power{X}{c}$,
respectively.  
Let $f\colon w_1\ra w_2$ and $g\colon w_2\ra w_3$.
\begin{enumerate}[(a)]
\item $[\id_{w_3}\tens f]_r=[f]_r$ and $[f\tens \id_{w_3}]_l=[f]_l$,
\item $[gf]_r=[g]_r\cdot [f]_r$.  
\end{enumerate}
\end{prop}

\begin{proof}
For part (a) first note that $\ab{\id_{w_3}\tens f}=\id_c\tens
\ab{f}$, by an easy argument.  Next use the following diagram, where
we have suppressed some tensor signs for typograhical reasons:
\[\xymatrixcolsep{2.3pc}
\xymatrix{
\power{X}{-b-c}\power{X}{a+c} \ar[r]^-\phi &
\power{X}{-b-c}\power{X}{c}\power{X}{a} \ar[d]^{\phi\tens
\id}\ar[r]^{1\tens 1\tens \ab{f}} &
\power{X}{-b-c}\power{X}{c}\power{X}{b} \ar[d]^{\phi\tens
\id}\ar[r]^{\id\tens \phi} &
\power{X}{-b-c}\power{X}{b+c} \ar[d]^\phi \\
\power{X}{a-b} \ar[u]^\phi\ar[r]^\phi & \power{X}{-b}\power{X}{a}
\ar[r]^{1\tens \ab{f}} &
\power{X}{-b}\power{X}{b} \ar[r]^\phi & S.
}
\]
The three squares are readily checked to commute; in the case of the 
outer ones this is by Theorem~\ref{th:multi-coherence1}.  The
composite from $\power{X}{a-b}$ to $S$ across the `top' of the diagram
is $[\id_{w_3}\tens f]_r$,  and
the composite across the bottom is $[f]_r$.  The argument showing
$[f\tens \id_{w_3}]_l=[f]_l$ is entirely similar.

For (b) we first examine the commutative diagram
\[ \xymatrixcolsep{2.5pc}\xymatrix{
\power{X}{a-c} \ar[r]^-{\phi_1} & \power{X}{b-c}\tens\power{X}{a-b}
\ar[r]^{\id\tens [f]_r} \ar[dr]_{[g]_r\tens[f]_r}&
\power{X}{b-c}\tens S\ar[r]^-{\phi_2}\ar[d]^{[g]_r\tens \id} & \power{X}{b-c}\ar[d]^{[g]_r} \\
&&S\tens S\ar[r] & S.
}
\]
Let $H=\phi_2\circ (\id\tens [f]_r)\circ \phi_1$.  The composite
across the `bottom' of the diagram is $[g]_r\cdot [f]_r$, so we have
$[g]_r\cdot [f]_r=[g]_r\circ H$.  

Next consider the following diagram:
\[ \xymatrixcolsep{2.4pc}\xymatrix{
\power{X}{a-c}\ar[r]^-\phi \ar[d]_{\phi_1}& \power{X}{-c}
\power{X}{a}\ar[r]^{\id\tens \ab{f}}
&\power{X}{-c} \power{X}{b} \ar[r]^{\id\tens \ab{g}} &
\power{X}{-c} \power{X}{c}\ar[r]^-\phi & S \\
\power{X}{b-c}\power{X}{a-b}\ar[r]^-{\id\tens \phi}
& \power{X}{b-c}\power{X}{-b} \power{X}{a} \ar[u]_{\phi\tens \id}
\ar[r]^{1\tens 1\tens \ab{f}} &
\power{X}{b-c} \power{X}{-b} \power{X}{b}\ar[u]^{\phi\tens
\id}\ar[r]^-{\id\tens \phi}
&\power{X}{b-c}S \ar[r]_-{\phi_2}&\power{X}{b-c}\ar[u]^{[g]_r}\ar[ull]^\phi
}
\]
All of the regions of the diagram commute, in two cases by
Theorem~\ref{th:multi-coherence1}.  The composite across the top
row is $[gf]_r$.  The composite across the bottom edge from
$\power{X}{a-c}$ to $\power{X}{b-c}$ is the map $H$.  So the diagram
shows that $[gf]_r=[g]_r\circ H$, and the latter equals $[g]_r\cdot
[f]_r$ by the preceding paragraph.   
\end{proof}

\begin{remark}
It is informative to check that the argument for part (b) of
Proposition~\ref{pr:new} does not dualize to prove $[gf]_l=[g]_l\cdot
[f]_l$.  The reason comes down to the fact that the formula $g\tens
f=(g\tens \id)\circ (\id\tens f)$ has the identity tensored on the {\it left\/}
side of the $f$.  The dual argument shows $[gf]_l=[f]_l\cdot [g]_l$,
although we will not need this fact.  
\end{remark}

Our next task is to focus on the case where $f\colon w_1\ra w_2$ and
$w_1\iso w_2\iso \power{X}{a}$.  
Note that in this case $\ab{f}$ is a map $\power{X}{a}\ra
\power{X}{a}$ and so we also have the invariants $D(\ab{f})$ and
$\tr(\ab{f})$, which like $[f]_l$ and $[f]_r$ are elements of $\pi_0(S)$.  
The following result gives the relation between all of these
constructions:

\begin{prop}
\label{pr:self-map}
Let $w_1$ and $w_2$ be two words that are formally isomorphic to
$\power{X}{a}$, for $a\in \Z^n$.  Let $f\colon w_1\ra w_2$ be a map.  
\begin{enumerate}[(a)]
\item $[f]_r=[f]_l=D(\ab{f})$, and $\tr(\ab{f})=D(f)\cdot
\tr(\id_{\power{X}{a}})$.  
\item $[\id_c\tens f]_r=[f\tens\id_c]_r=[f]_r$, for any $c\in \Z^n$.
\item For a canonical isomorphism $\phi\colon w_1\ra w_2$ (as provided
by Theorem~\ref{th:multi-coherence1}) one has $[\phi]_r=\id_S$.  
\end{enumerate}
\end{prop}

Before giving the proof let us introduce one more important
definition.  For $a,b\in \Z^n$ we have the twist map $t_{a,b}\colon
\power{X}{a}\tens \power{X}{b} \ra \power{X}{b}\tens \power{X}{a}$.
We write $T_{a,b}=\ab{t_{a,b}}$, which is a map $\power{X}{a+b}\ra
\power{X}{a+b}$.  It is easy to check that $T_{a,b}\circ
T_{b,a}=\id$.  Note that $T_{a,-a}$ is a map $S\ra S$.  

\begin{proof}[Proof of Proposition~\ref{pr:self-map}]
For (a) we consider the following diagram:
\[ \xymatrix{
S \ar[dr]\ar[r] & \power{X}{-a}\tens \power{X}{a} \ar[r] \ar[d]^{t_{-a,a}} & \power{X}{-a}\tens w_1
\ar[r]^{\id\tens f} \ar[d]^t &\power{X}{-a}\tens w_2 \ar[r]\ar[d]^t & \power{X}{-a}\tens
\power{X}{a} \ar[d]_{t_{-a,a}}\ar[r] & S \\
& \power{X}{a}\tens \power{X}{-a} \ar[r] & w_1\tens \power{X}{-a}
\ar[r]^{f\tens \id} & w_2\tens \power{X}{-a} \ar[r] & \power{X}{a}\tens
\power{X}{-a} \ar[ur]
}
\]
where all unlabelled maps are canonical isomorphisms (i.e., they
should be labelled with $\phi$).  The three squares are commutative,
but the triangles on the two ends are not; the automorphism of $S$
obtained by moving around one of these triangles is either $T_{a,-a}$
or $T_{-a,a}$, depending on which direction the composite is taken.
The composite across the top of the diagram is $[f]_r$ and the composite across the
bottom is $[f]_l$.  The diagram thus yields the formula
\[ [f]_r = T_{a,-a} \circ [f]_l \circ T_{-a,a}.
\]
But this formula takes place in the monoid $\End(S)$, which is
commutative by Proposition~\ref{pr:inv-end}.  So we obtain
$[f]_r=[f]_l\circ T_{a,-a}\circ T_{-a,a}=[f]_l\circ \id_S=[f]_l$.

Next consider the diagram
\[ \xymatrixcolsep{2.3pc}\xymatrix{
S\tens \power{X}{a} \ar[d]_{\phi}\ar[r]^-{\phi\tens \id} 
& \power{X}{a} \tens \power{X}{-a} \tens
\power{X}{a} \ar[r]^{\phi\tens\id\tens \id}\ar[dl] & w_1 \tens \power{X}{-a} \tens \power{X}{a}
\ar[r]^-{f\tens \id\tens \id}
\ar[dl] & w_2\tens \power{X}{-a} \tens \power{X}{a}
\ar[dl]\ar[d]^-{\phi\tens\id\tens \id}\\
\power{X}{a}\tens S \ar[r]_-{\phi\tens\id} & w_1\tens S \ar[r]_{f\tens\id} & w_2\tens S\ar[d]_{\phi\tens\id} & 
\power{X}{a}\tens \power{X}{-a}\tens \power{X}{a}
\ar[d]^-{\phi\tens\id\tens\id}\ar[dl] \\
&&\power{X}{a}\tens S \ar[r]_{\phi}& S\tens \power{X}{a}.
}
\]
The diagonal maps are all equal to the identity on the left tensor
factor and the canonical isomorphism $\phi\colon \power{X}{a}\tens
\power{X}{-a}\ra S$ on the other two factors.  All of the `squares'
obviously commute in the diagram, and the triangles on the two ends
commute by Theorem~\ref{th:multi-coherence1} since all the maps
are canonical isomorphisms.  The composition across the `top' of the
diagram equals $[f]_l\tens \id_{\power{X}{a}}$.  Condensing the
diagram to its outer rim yields the commutative square
\[ \xymatrix{
S\tens \power{X}{a} \ar[r]^{[f]_l\tens \id}\ar[d]_\phi &
S\tens \power{X}{a} \ar[d]^\phi \\
\power{X}{a}\tens S \ar[r]^{\ab{f}\tens \id_S} & \power{X}{a}\tens S.
}
\]
By Lemma~\ref{le:trace}(a) the top and bottom maps have the same
$D$-invariant, and the $D$-invariant of the bottom map is also that of
$\ab{f}$ (using the unital isomorphism).  Finally,  $D([f]_l\tens \id)=D([f]_l)=[f]_l$ by
Lemma~\ref{le:trace}(b).  
This ends the
proof of (a).

Part (b) follows immediately from (a) and Proposition~\ref{pr:new}(a).
Part (c) is a consequence of coherence: the composite
\[ \power{X}{a}\llra{\phi_1} w_1 \llra{\phi} w_2 \llra{\phi_2}
\power{X}{a} 
\]
is a canonical map and must therefore equal the identity by
Theorem~\ref{th:multi-coherence1}.  So $\ab{\phi}=\id_{\power{X}{a}}$
and 
$[\phi]_r=D(\id_{\power{X}{a}})=\id_S$.  
\end{proof}

The elements $T_{a,b}\in \Aut(S)$ are, of course, ubiquitous in
calculations.  We define
\[ \tau_{a,b}=[t_{a,b}]_l=[t_{a,b}]_r=D(T_{a,b})\in \pi_0(S),
\]
where the second two equalities are by Proposition~\ref{pr:self-map}.
Recall that we have the basic commuters $\tau_i=\tr(\id_{X_i})\in
\pi_0(S)$ and these satisfy $\tau_i^2=1$ by
Proposition~\ref{pr:tr-inv}.  Recall as well that
$\tau_i=D(t_{X_i,X_i})$ by Proposition~\ref{pr:tr=D}.  If
$e_1,\ldots,e_n$ is the standard basis for $\Z^n$ then this just says that
$\tau_i=\tau_{e_i,e_i}$.
Let us also point out that if $i\neq j$ then $\tau_{e_i,e_j}=\id_S$;
in fact $T_{e_i,e_j}$ is the composite
\[ \power{X}{e_i+e_j}\llra{\phi} X_i\tens X_j \llra{t_{X_i,X_j}}
X_j\tens X_i \llra{\phi} \power{X}{e_i+e_j}
\]
and this equals the identity map either by
Theorem~\ref{th:multi-coherence2} or by just looking at the definitions.
Quite generally,
we can express all of the elements $\tau_{a,b}$ in terms of the basic
commuters:

\begin{prop}
\label{pr:tau-ab}
For all $a,b\in \Z^n$ one has $\tau_{a,b}=\tau_1^{(a_1b_1)}\cdots
\tau_n^{(a_nb_n)}$. 
\end{prop}

\begin{proof}
Recall that $T_{a,b}$ is the composite $\power{X}{a+b}\llra{\phi}
\power{X}{a}\tens \power{X}{b} \llra{t_{a,b}} \power{X}{b}\tens
\power{X}{a} \llra{\phi} 
\power{X}{a+b}$.  Observe that we can also obtain this map as a long
composite
\begin{myequation}
\label{eq:tab1}
 \power{X}{a+b}\ra w_1 \ra w_2 \ra \cdots \ra w_N \ra \power{X}{a+b}
\end{myequation}
where each $w_k$ is a tensor word in the $X_i^{\pm 1}$'s and each map
is either
\begin{enumerate}[(1)]
\item a canonical isomorphism $\phi$ (as provided by
Theorem~\ref{th:multi-coherence1}),
\item a tensor product of $t_{X_i,X_i}$ with identity maps, 
\item a tensor product of $t_{X_i,X_i^{-1}}$ with identity maps, or
\item a tensor product of $t_{X_i^{-1},X_i^{-1}}$ with identity maps.
\end{enumerate}
In the `standard'
way to obtain such a composite the number of transpositions of types
(2)--(4) will be $|a_ib_i|$, for any chosen value of $i$.  
Let $f\colon S\ra S$ be the map $\prod_i
\tau_i^{(a_ib_i)}$, noting that only the parity of $a_ib_i$ matters in
the exponent since $\tau_i^2=\id_S$.  Consider the composite
\begin{myequation}
\label{eq:tab2} \power{X}{a} \llra{\phi} S\tens \power{X}{a} \llra{f\tens \id}
S\tens \power{X}{a} \llra{\phi}\power{X}{a}.
\end{myequation}
It follows from Theorem~\ref{th:multi-coherence2} that the
composites in (\ref{eq:tab1}) and (\ref{eq:tab2}) are the same,
because by construction they have the same $i$-parity for every $i$.  
Consequently, the $D$-invariant of the two composites  is the same.  But the
$D$-invariant of (\ref{eq:tab2}) is manifestly equal to the map $f$.
We have thus proven that $f=D(T_{a,b})=\tau_{a,b}$.  
\end{proof}

\begin{remark}
\label{re:tau}
As a consequence of Proposition~\ref{pr:tau-ab} and the fact that
$\tau_i^2=1$ note that we have
$\tau_{a,b}\tau_{a,c}=\tau_{a,b+c}=\tau_{a,b-c}$.  Likewise,
$\tau_{a,b}=\tau_{-a,b}=\tau_{b,a}$.   Identities such as these
will often be used.
\end{remark}

Before proceeding further we need a lemma, which is easy but worth recording:

\begin{lemma} 
\label{le:dumb}
Consider composable maps $\power{X}{a} \llra{f}
\power{X}{a} \llra{g} S \llra{h} S$.  Then as elements of $\pi_*(S)$
one has $hg=h\cdot g=g\cdot h$ and $gf=g\cdot D(f)$.  More generally, we can
write $hgf=h\cdot g\cdot D(f)=g\cdot h\cdot D(f)$.  
\end{lemma}

\begin{proof}
We have already seen that $h\cdot g=g\cdot h$, in
Proposition~\ref{pr:R}.  
The identificaton of these with the composite $hg$ is easy.  For
the second identity we have
\[ g\circ f=g\circ (\id_X\ctens D(f))=g\ctens D(f)=g\cdot D(f)
\]
where the second equality is from  Lemma~\ref{le:ctens} and the third
equality follows from the definitions.
Finally, the identity for $hfg$ is a consequence of the previous identities.
\end{proof}

Now we can move on to the study of $[f]_r$ and $[f]_l$ for general
maps $f$.

\begin{prop}
\label{pr:l-r-general}
Let $w_1$ and $w_2$ be two tensor words, where $w_1$ is formally
isomorphic to $\power{X}{a}$ and
$w_2$ is formally isomorphic to $\power{X}{b}$.  Let $f\colon w_1\ra w_2$, and let $c\in
\Z^n$.
Write $\id_c$ for $\id_{\power{X}{c}}$.  
\begin{enumerate}[(a)]
\item $[f]_r=[f]_l \cdot \tau_{b,a-b}$
\item $[\id_c\tens f]_r=[f]_r$
\item $[f\tens \id_c]_r=[f]_r \cdot \tau_{a-b,c}$
\item $[f\tens\id_c]_l=[f]_l$
\item $[\id_c\tens f]_l=[f]_l\cdot \tau_{a-b,c}$
\item If $g\colon w_2\ra w_3$ where $w_3\iso \power{X}{c}$, then
$[gf]_r=[g]_r\cdot [f]_r$ and likewise $[gf]_l=[g]_l\cdot
[f]_l\cdot \tau_{a-b,c-d}$.
\item Let $g\colon w_1'\ra w_2'$ where $w_1'\iso \power{X}{c}$ and
$w_2'\iso \power{X}{d}$. Then 
\[ [f\tens g]_r = [f]_r\cdot [g]_r\cdot
\tau_{a-b,d}=[g]_r\cdot[f]_r\cdot \tau_{a-b,c} \]
and
\[ [f\tens g]_l=[f]_l\cdot [g]_l\cdot \tau_{b,c-d}=[g]_l\cdot
[f]_l\cdot \tau_{a,c-d}.
\]
\end{enumerate}  
\end{prop}

\begin{proof}[Proof of Proposition~\ref{pr:l-r-general}]
Note first that parts (b), (d), and the first part of (f) were already proven in
Proposition~\ref{pr:new}; they are only restated here for ease of
reference.  Note also that parts (c), (e), and the second part of (f) 
are formal consequences of the aforementioned results, using (a).  So
most everything follows from (a).

To prove (a) we consider the usual diagram
\[ \xymatrix{
\power{X}{a-b}\ar[r]^-\phi \ar[dr]_\phi & \power{X}{-b}\tens \power{X}{a}
\ar[d]^{t_{-b,a}} \ar[r]^{\id\tens \ab{f}} & \power{X}{-b}\tens
\power{X}{b} \ar[d]_{t_{-b,b}} \ar[r]^-\phi & S \\
& \power{X}{a}\tens \power{X}{-b} \ar[r]^{\ab{f}\tens \id}
& \power{X}{b} \tens \power{X}{-b}. \ar[ur]_\phi
}
\]
The square commutes but the triangles do not;
the composite across the top is $[f]_r$ and the composite across the
bottom is $[f]_l$.  The diagram yields the formula
\[ [f]_r=T_{b,-b}\circ [f]_l\circ T_{-b,a}.
\]
From Lemma~\ref{le:dumb} we get that in $\pi_*(S)$ one has the formula
\[ [f]_r=[f]_l\cdot
T_{b,-b}\cdot \tau_{-b,a}=
[f]_l\cdot
\tau_{b,-b}\cdot \tau_{-b,a}=
[f]_l\cdot \tau_{a-b,b}
\] 
(using
Remark~\ref{re:tau} for the final equality).

Finally, in (g) we simply use that $f\tens g=(\id_{w_2}\tens
g)\circ(f\tens \id_{w_1'})=(f\tens \id_{w_2'})\circ (\id_{w_1}\tens
g)$.  
The desired formulas follow from the combined application of the
previous parts.
\end{proof}

\subsection{Skew-commutativity}
Skew-commutativity for $\pi_*(S)$ follows immediately from the various
formulas in Proposition~\ref{pr:l-r-general}(g).  We give a slightly
more general version here:

\begin{prop}
\label{pr:skew}
Let $W$ be an object in $\cC$, let $f\colon \power{X}{a}\ra S$ and
$g\colon \power{X}{b}\ra W$.  Then under the left and right actions of
$\pi_*(S)$ on $\pi_*(W)$ we have 
\[ f\cdot g=g\cdot f \cdot
\tau_{a,b}=g\cdot f \cdot \tau_1^{(a_1b_1)}\cdots\tau_n^{(a_nb_n)}.
\]  
\end{prop}

\begin{proof}
Consider the diagram
\[ \xymatrix{
\power{X}{a+b} \ar[r]^-\phi\ar[dr]_{\phi} & \power{X}{a}\tens
\power{X}{b} \ar[d]^{t_{a,b}} \ar[r]^{f\tens g} & S\tens W \ar[r]\ar[d]_{t_{S,W}} & W \\
& \power{X}{b}\tens\power{X}{a} \ar[r]^{g\tens f}& W\tens S\ar[ur]
}
\]
and note that all regions commute except the leftmost triangle.  The
composite across the top is $f\cdot g$, and the composite across the
bottom is $g\cdot f$.  The diagram yields the identity
$f\cdot g=(g\cdot f)\circ A$
where $A$ is the appropriate self-map of $\power{X}{a+b}$ coming from
the left triangle.  By Lemma~\ref{le:dumb} we obtain $f\cdot g=g\cdot
f\cdot D(A)$ in $\pi_*(W)$, and we know that $D(A)=[t_{a,b}]_r=\tau_{a,b}$.  The
identification of $\tau_{a,b}$ as $\prod_i \tau_i^{(a_ib_i)}$ is from
Proposition~\ref{pr:tau-ab}.
\end{proof}

\begin{remark}
There are other settings in which one can prove similar
skew-commutativity results.  For example, if $W$ is a commutative
monoid in $\cC$ (with respect to $\tens$) then $\pi_*(W)$ has the same
skew-commutativity law as $\pi_*(S)$, where now the $\tau_i$'s are
regarded as elements of $\pi_*(W)$ via the unit map $S\ra W$.
If $Z$ is a bimodule over $W$ then  there is a corresponding 
skew-commutativity result in that setting.  All of the proofs are the
same as for Proposition~\ref{pr:skew} above, so we leave these to the reader.  
\end{remark}

\section{More general grading schema}
\label{se:schema}

Let $(\cC,\tens,S)$ be an additive category with a symmetric monoidal
structure that is additive in each variable.   
Let $A$ be a finitely-generated abelian group, and fix a homomorphism
$h\colon A\ra \Pic(\cC)$.  For each $a\in A$ let $X_a$ be a
chosen object in the isomorphism class $h(a)$; assume $X_0=S$.  For $W$ in $\cC$
define $\pi_*^A(W)$ to be the $A$-graded abelian group $a\mapsto
\cC(X_a,W)$.  To obtain a product on $\pi^A_*(S)$ one can
start by choosing isomorphisms
\[ \sigma_{a,b}\colon X_{a+b} \ra X_a\tens X_b\]
for each $a,b\in A$.  If $f\colon X_a \ra S$ and $g\colon X_b\ra S$
then we define the product $f\cdot g$ to be the composite
\[ X_{a+b} \llra{\sigma_{a,b}} X_a\tens X_b \llra{f\tens g} S\tens S
\iso S.
\]
This clearly defines a distributive product on $\pi^A_*(S)$.  The
questions that arise are:
\begin{enumerate}[(1)]
\item Is it possible to choose the $\sigma_{a,b}$ isomorphisms so that
the product on $\pi^A_*(S)$ is associative and unital?
\item If there are multiple ways to accomplish (1), do they give rise
to isomorphic rings?  That is, is the ring structure on
$\pi^A_*(S)$ in some sense canonical?
\end{enumerate}
Note that in Section~\ref{se:app} we proved that the answer to (1) is
yes in the case when $A$ is free.  The construction depended on
choosing a free basis $e_1,\ldots,e_n$ for $A$ and then fixing a
specific choice of isomorphism $\alpha_i\colon S\ra X_{-e_i}\tens
X_{e_i}$ for each $i$; so the construction was certainly not canonical.  

We will see below that the answer to (1) is yes in general, but the
answer to (2) is no.  In fact, the set of isomorphism classes of
ring structures obtained in this way is parameterized by the
cohomology group $H^2(A;\Aut(S))$.  Much of the material behind this
story seems to be standard, but we were unable to find an adequate reference
(the introduction to the paper \cite{CK} gives a partial survey, though).

I am grateful to Victor Ostrik and Vadim Vologodsky for conversations
about the results in this section.

\medskip

Let us call the collection $(\sigma_{a,b})_{a,b\in A}$ an
\mdfn{$A$-trivialization} of $\cC$ with respect to $X$ if it satisfies
two properties:
\begin{enumerate}[(1)]
\item For every $a\in A$ the isomorphisms $\sigma_{a,0}$ and
$\sigma_{0,a}$ coincide with the unital isomorphisms in $\cC$.  
\item For every $a,b,c\in A$ the following pentagon commutes:
\[ \xymatrixrowsep{1pc}\xymatrix{
X_{a+b+c}\ar[r]^{\sigma_{a,b+c}}\ar[dd]_{\sigma_{a+b,c}} & X_a\tens X_{b+c}
\ar[dr]^{\id\tens \sigma_{b,c}} \\
&& X_a\tens (X_b\tens X_c)\\
X_{a+b}\tens X_c \ar[r]^-{\sigma_{a,b}\tens \id} & (X_a\tens X_b)\tens X_c \ar[ur]^a
}
\]
\end{enumerate}
Under conditions (1) and (2) the induced product on $\pi_*^A(S)$ is
both associative and unital; we will call this the 
\dfn{standard ring structure} on $\pi^A_*(S)$ associated to
$\sigma$.  Note that there possibly exist ring structures on
$\pi_*^A(S)$ which are not standard---i.e., which do not arise from an
$A$-trivialization.  Such structures are not part of the theory we
develop here.

Given two $A$-trivializations $\sigma$ and $\sigma'$ we get two ring
structures $\pi_*^A(S)_\sigma$ and $\pi_*^A(S)_{\sigma'}$.  
Are the two standard rings obtained in this way isomorphic?  The question is
not easy to answer when stated so broadly, but we can refine it
somewhat.  The evident way to construct a map 
$\pi_*^A(S)_\sigma \ra \pi_*^A(S)_{\sigma'}$ would be to send
each $f\colon X_a\ra S$ to $f\cdot u(a)$ for some chosen $u(a)\in\Aut(S)$
that is independent of $f$.
[Note that it does not matter which product we use for
$f\cdot u(a)$, since both the $\sigma$-product and the
$\sigma'$-product will give the same answer if one of the factors lies
in $\pi_0^A(S)$, by condition (1).]  Let us say that a \dfn{standard isomorphism}
between standard ring structures is one that is of this form; note
that it is determined by  a chosen map of sets $u\colon A\ra \Aut(S)$.  

Here is the main goal of this section:

\begin{prop}
\label{pr:A}
Suppose $(\cC,\tens,S)$, $h\colon A\ra \Pic(\cC)$, and $X\colon A \ra
\ob(\cC)$ are as in the beginning of this section.  
\begin{enumerate}[(a)]
\item There exists
an $A$-trivialization of $\cC$ with respect to $X$, and therefore a
resulting standard ring structure on $\pi_*^A(S)$.  
\item The set of all
such $A$-trivializations is in bijective correspondence with
$Z^2(A;\Aut(S))_{norm}$.  
\item
The set of
different possible standard ring structures on $\pi_*^A(S)$, up to
standard isomorphism, is in bijective correspondence with
$H^2(A;\Aut(S))$.
\end{enumerate}
\end{prop}

Note that we essentially already encountered this in the case where
$A$ was $\Z^n$.  In that case $H^2(A;\Aut(S))\iso \Aut(S)^n$
(non-canonically), and one only obtains a ring structure after fixing
a basis for $A$ together with $n$ elements of $\Aut(S)$---as we found
in our earlier treatment.  The overall lesson is that grading morphism
sets by invertible objects is a bit dicey when it comes to product
structures; the rings obtained are typically neither unique nor
canonical.

\begin{remark}
One can also ask about the graded-commutativity properties
of $\pi_*^A(S)$.  It is easy to prove that if $f\in \pi_a^A(S)$ and
$g\in \pi_b^A(S)$ then $fg=gf\cdot
\theta_{a,b}$ where $\theta_{a,b}=D(\sigma_{b,a}^{-1}\circ
t_{a,b}\circ \sigma_{a,b})\in \Aut(S)$.
 We have not explored the properties of $\theta\colon A^2\ra
\Aut(S)$, mostly due to a lack of application.  Our analysis in the free case
(Proposition~\ref{pr:tau-ab}) suggests this might be a nice exercise. 
\end{remark}

We will prove Proposition~\ref{pr:A} by analyzing a very specific
class of monoidal categories, and then reducing to that case.

\subsection{Monoidal categories of type $\mdfn{(A,N)}$}
\label{se:ANm}
Fix abelian groups $A$ and $N$.  Let $\cC=\cC[A,N]$ be the category with object
set $A$, where there are no maps between distinct objects, and where
the set of self-maps of each object is equal to $N$.
Define a bifunctor $\tens \colon \cC\times \cC\ra \cC$ whose behavior
on objects is given by the sum in $A$, and whose behavior on morphisms
is given by the sum in $N$.  To equip $(\cC,\tens,0_A)$
with a monoidal structure we must specify unital
isomorphisms $a\oplus 0\iso a$ and $0\oplus a\iso a$; but $a\oplus
0=a=0\oplus a$, so we can (and will) just take the isomorphisms to be
the identities.  

We must also specify, for every $a,b,c\in A$, an
associativity isomorphism $\alpha_{a,b,c}\colon (a\tens b)\tens c \ra
a\tens (b\tens c)$.  Again, since the objects $(a\tens b)\tens c$ and $a\tens
(b\tens c)$ are actually equal (they both are equal to the object $a+b+c$)
we are just specifying an element $\alpha_{a,b,c}\in N$.  We {\it
could\/} require this to be the identity, but we wish to not be so
restrictive here.  Let us call a monoidal structure on
$(\cC,\tens,0_A)$ obtained in this way an \dfn{extended} monoidal
structure, as it is an extension of the canonical tensor functor and
unital isomorphisms.  

The pentagonal condition that a monoidal structure must satisfy just
says that $\alpha\colon A^3\ra N$ is a $3$-cocycle in the usual bar
complex $C^*(A;N)$ for computing group cohomology.  Compatibility
between associativity and unital isomorphisms then requires that
$\alpha_{a,b,c}=0$ if any of $a$, $b$, or $c$ are zero; in other
words, we have a normalized cocycle.  In this way we see that extended
monoidal structures on $(\cC,\tens,0)$ are in bijective correspondence
with the group $Z^3(A;N)_{norm}$.  Even more, it is easy to see that
elements of the group $H^3(A;N)$ are in bijective corresondence with
extended monoidal structures on $(\cC,\tens,0)$ up to isomorphism.  
This is a standard story.  For $\alpha\in Z^3(A;N)_{norm}$ write
$\cC_\alpha=\cC[A,N]_\alpha$ for the corresponding monoidal category.

Fix an element $\alpha\in Z^3(A;N)_{norm}$.  In $\cC_\alpha$ let us
ask if there is an $A$-trivialization with respect to the identity
map: that is, do there exist isomorphisms $\sigma_{a,b}\colon a\tens
b\ra a+b$ satisying the required associativity and unital conditions?
Again, $\sigma_{a,b}$ is just an element of $N$ and so
$\sigma\in C^2(A;N)$.  The unital condition is the requirement
$\sigma\in C^2(A;N)_{norm}$ and the associativity condition translates
to $\delta\sigma=\alpha$.  So the cohomology class of $\alpha$ in
$H^3(A;N)$ (or $H^3(A;N)_{norm}$, which is the same thing) is the
obstruction to the existence of the desired $\sigma$'s.

\begin{remark}
Note that the $\sigma_{a,b}$'s are giving a (strong) monoidal structure on the
identity functor $\cC[A,N]_0\ra \cC[A,N]_\alpha$, showing that the
domain and target are monoidally equivalent.  This is why we call 
the collection $(\sigma_{a,b})_{a,b\in A}$ a {\it trivialization}
of the monoidal structure $\cC[A,N]_\alpha$.
\end{remark}

\subsection{Symmetric monoidal categories of type $\mathbf{(A,N)}$}
There is a similar story for the existence of extended {\it symmetric\/} monoidal
categories on $\cC$.  Here one must specify both the $\alpha_{a,b,c}$
elements and certain elements $\beta_{a,b}\in N$ giving the
commutativity isomorphisms.  One again finds that the set of extended
structures is in bijective correspondence with the $3$-cocycles in a
certain complex.  To describe this, let $\cE$ be the complex
\[ \Z\langle A^4\rangle \oplus \Z\langle A^3\rangle_1 \oplus
\Z\langle A^3\rangle_2 \oplus \Z\langle
A^2\rangle \llra{d_4} \Z\langle A^3\rangle \oplus \Z\langle A^2\rangle
\llra{d_3} \Z\langle A^2\rangle \llra{d_2} \Z\langle A\rangle\llra{d_1} 0
\]
(concentrated in homological degrees $0$ through $4$) with
differentials defined on free generators by the formulas below:
\begin{align*}
&d_1([a])=0, \quad d_2([a|b])=[a]-[a+b]+[b],\\
&d_3([a|b|c])=[b|c]-[a+b|c]+[a|b+c]-[a|b],\quad
d_3([a|b])=[a|b]-[b|a]\\
&d_4([a|b|c|d])=[b|c|d]-[a+b|c|d]+[a|b+c|d]-[a|b|c+d]+[a|b|c] \\
& d_4([a|b|c]_1)=[a|b|c]-[a|c|b]+[c|a|b]-[b|c]+[a+b|c]-[a|c],\\ 
& d_4([a|b|c]_2)=[a|b|c]-[b|a|c]+[b|c|a]+[a|b]-[a|b+c]+[a|c],\\ 
& d_4([a|b])=[a|b]+[b|a].
\end{align*}
Let $D\subseteq \cE$ be the ``degenerate'' subcomplex spanned by all
symbols $[a_1|\ldots|a_n]$ in which at least one of the $a_i$'s is
zero, and note that this is indeed closed under the differential.  
A little legwork shows that extended symmetric monoidal structures on
$\cC$ correspond to normalized $3$-cocycles $(\alpha,\beta)\in
Z^3(\Hom(\cE,N))$ (where `normalized' refers to cocycles that vanish
on the degenerate subcomplex).  

\begin{remark}
The paper \cite{H} used a similar complex but where the $\Z\langle
A^3\rangle_2$ term was omitted from $\cE_4$ (and where the grading of
the complex was shifted by $1$).  It is easy to see that omitting this
term does not effect $H_3(\cE)$ or $Z^3(\cE;N)$; in effect, the relations coming
from this term are consequences of the ones coming from
$d_4([a|b|c]_1)$ and $d_4([a|b])$, by an easy exercise.  We are using
the larger complex because it allows us to directly quote 
published results from \cite{EM2}.  
\end{remark}

The complex $\cE$ was introduced by Eilenberg and MacLane \cite{EM1, EM2}: it is the
first few terms of their iterated bar construction.  They prove that
their complex calculates the  homology of Eilenberg-MacLane
spaces in the stable range; in particular, 
\[ H_i(\cE)\iso H_{n+i-1}(K(A,n)) \]
for $1\leq i\leq 3$ and $n\geq 3$ \cite[Theorem 6]{EM1}.  Let us write
$H_*^{EM}(A)$ for $H_*(\cE)$ and $H^*_{EM}(A;N)$ for
$H^*(\Hom(\cE,N))$.  Eilenberg and MacLane calculated the following:

\begin{prop}[Eilenberg-MacLane]\mbox{}\par
\label{pr:EM}
\begin{enumerate}[(a)]
\item
There are natural isomorphisms 
\[ H_1^{EM}(A)\iso A, \quad H_2^{EM}(A)\iso
0, \quad\text{and}\quad H_3^{EM}(A)\iso A/2A.
\]  
The last isomorphism is induced by
$a\mapsto [a|a]$.
\item
There is an isomorphism $H^3_{EM}(A;N)\ra
\Hom(A/2A,N)=\Hom(A,\psub{2}{N})$ given by $(\alpha,\beta)\mapsto
[x\mapsto \beta(x,x)]$. 
\end{enumerate}   
\end{prop}

The isomorphisms in part (a) are from \cite[Theorems 20.3, 20.5,
23.1]{EM2}. 
Note that (b) is an immediate consequence of (a), using the Universal
Coefficient Theorem.  Also, note that part of the claim in (b) is that
if $(\alpha,\beta)$ is a $3$-cocyle in $\Hom(\cE,N)$ then $x\mapsto
\beta(x,x)$ is linear and takes its values in $\psub{2}{N}$.  Neither
of these claims is immediately obvious, although they follow from (a).
Separate from this, however, observe that they also follow from 
Proposition~\ref{pr:tau} because $x\mapsto \beta(x,x)$ is the
$\tau$-function for the symmetric monoidal category
$\cC[A,N]_{(\alpha,\beta)}$.

\medskip
Observe  that the bar complex $C_*(A)$ is contained inside $\cE$ as a
subcomplex.  Let $Q$ be the quotient, so that we have the short exact
sequence $0\ra C_*(A) \ra \cE \ra Q\ra 0$.  Note that $Q$ has the form
$\Z\langle A^3\rangle_1\oplus \Z\langle A^3\rangle_2 
\oplus \Z\langle A^2\rangle \ra \Z\langle
A^2\rangle$, concentrated in degrees $3$ and $4$.  Applying
$\Hom(\blank,N)$, the long exact sequence in cohomology then gives
\begin{myequation}
\label{eq:les}
\qquad\qquad \cdots \la H^3(A;N) \la H^3_{EM}(A;N) \la H^3(Q;N) \la H^2(A;N)\la \cdots
\end{myequation}
The group $H^3(Q;N)$ is easy to analyze: it is the collection of
$\beta\colon A^2\ra N$ satisfying $\beta(x,y)=-\beta(y,x)$ and
$\beta(y,z)-\beta(x+y,z)+\beta(x,z)=0$ for all $x,y,z\in A$.  In other
words, $H^3(Q;N)$ is the collection of alternating bilinear forms
$A\times A\ra N$; write this as $H^3(Q;N)\iso \AltBilin(A,N)$.  The
map $\AltBilin(A,N)\ra H^3_{EM}(A;N)$ sends an alternating form
$\beta$ to the cohomology class $[(0,\beta)]$.  

The following lemma is the key calculation of this entire section:

\begin{lemma}
\label{le:trivial}
For any abelian groups $A$ and $N$, 
the map $H^3_{EM}(A;N)\ra H^3(A;N)$ is the zero map.
\end{lemma}

\begin{proof}
We consider the commutative diagram
\[ \xymatrix{
\cdots \ar[r] & \AltBilin(A,N) \ar[r]^u\ar[dr]_{p} & H^3_{EM}(A,N) \ar[d]^\iso\ar[r]^v & H^3(A,N)
\ar[r] & \cdots\\
&& \Hom(A,\,\psub{2}{N})\ar@{=}[r] &\Hom(A/2A,N)
}
\]
where the top row is the long exact sequence (\ref{eq:les}), the vertical map is the
one from Proposition~\ref{pr:EM}(b), and the map $p$ is the
evident composite.  Note that $p$ sends an alternating bilinear form
$\theta\colon A\times A\ra N$ to the map $a\mapsto \theta(a,a)$.  But
it is easy to see that $p$ is surjective.  Indeed, pick elements
$\{e_i\}$ in $A$ whose mod $2$ reductions give a $\Z/2$-basis for
$A/2A$.  If $f\colon A/2A\ra N$ then define a bilinear form $b\colon
A\times A\ra N$ by $b(e_i,e_j)=0$ if $i\neq j$ and
$b(e_i,e_i)=f(e_i)$.  This is alternating because $2f(e_i)=0$, and $p(b)=f$.

Since $p$ is surjective it follows that $u$ is surjective, and so
$v=0$.  
\end{proof}

\subsection{Trivializations of $\mathbf{(A,N)}$-structures}

\begin{prop}
\label{pr:AN1}
Fix an extended symmetric monoidal structure $(\alpha,\beta)$ on
$\cC[A,N]$.    Then there exists a trivialization of the monoidal
structure $\cC[A,N]_\alpha$, and the set of all such trivializations
is in bijective correspondence with $Z^2(A;N)_{norm}$.  
\end{prop}

\begin{proof}
As we saw in (\ref{se:ANm}), a trivialization is simply an element 
$\sigma\in C^2(A;N)_{norm}$ satisfying $\delta
\sigma=\alpha$.  It is clear that if such a thing exists, the set of
all possibilities is in bijective correspondence with $Z^2(A;N)_{norm}$.  
To prove existence we need only show that $[\alpha]=0$ in $H^3(A;N)$.
But the map
$H^3_{EM}(A;N)\ra H^3(A;N)$
which sends $[(\alpha,\beta)]$ to $[\alpha]$ is the zero map by
Lemma~\ref{le:trivial}, so this finishes the proof.
\end{proof}

\subsection{The general case}
We will prove Proposition~\ref{pr:A} by reducing the construction of
an $A$-trivialization to the corresponding problem for an extended symmetric
monoidal category of type $(A,N)$.  This uses the
following lemma from 
\cite[Chapter II, Proposition 7]{H}:

\begin{lemma}
\label{le:symm-extended}
Let $(\cC,\tens,S)$ be a symmetric monoidal category in which every
object is invertible and every map is an isomorphism.  Then $\cC$ is
equivalent (as a symmetric monoidal category) to an extended symmetric
monoidal category of type $(\Pic(\cC),\Aut(S))$.
\end{lemma}

\begin{proof}
First recall that $\cC$ is equivalent to a symmetric monoidal category
where the associativity and unital conditions are strict \cite[Theorem
XI.3.1]{M}.  So we can
just assume that $\cC$ itself has these properties.

Let $N=\Aut(S)$.  
Let $\cD$ be the category whose objects are the elements of
$\Pic(\cC)$, where there are no maps between distinct objects, and
where every endomorphism of self-maps is equal to $N$.  

For each element $a\in \Pic(\cC)$ choose a fixed object $X_a$ in $\cC$
that belongs to this isomorphism class; when $a=[S]$ choose $X_a=S$.  
Moreover, for each $Y$ in
$\cC$ choose a fixed isomorphism $i_Y\colon Y\ra X_{[Y]}$.  
Define a functor $F\colon \cC\ra \cD$ by sending each object $Y$ to
its isomorphism class in $\Pic(\cC)$; if $g\colon Y_1\ra Y_2$ is a
map, then let $F(g)=D( i_{Y_2}\circ g\circ i_{Y_1}^{-1})$.  One
readily checks that this is indeed a functor, and that each self-map
$f\colon Y\ra Y$ is sent to its $D$-invariant $D(f)\in N$.

Likewise, define a functor $G\colon \cD\ra \cC$ by sending an object
$a\in \Pic(\cC)$ to $X_a$, and sending a self-map of $a$
corresponding to $n\in N$ to the unique self-map $X_a\ra X_a$ that has
$D$-invariant equal to $n$.  It is easy to check that $F$ and $G$ give
an equivalence of categories.  Note that $FG=\id_\cD$ and that
$G([S])=S$.  

Use the equivalence $(F,G)$ to transplant the symmetric monoidal
structure from $\cC$ onto $\cD$.  For example, define the monoidal
product on $\cD$ by
\[ d_1\tens d_2 = F(Gd_1 \tens Gd_2),
\]
and likewise for the associativity, unital, and commutativity
isomorphisms.  It is routine to check that the unit in $\cD$ is
strict, because this was assumed to be the case for $\cC$ and $G([S])=S$; in
contrast, the associativity isomorphisms need not be strict.  But one
readily verifies that this gives an extended symmetric monoidal
structure on $\cD$,  which is equivalent to $(\cC,\tens,S)$ by
construction.
\end{proof}

\begin{proof}[Proof of Proposition~\ref{pr:A}]
We begin by replacing $\cC$ by the subcategory $\cC^{inv}$ of invertible objects
and isomorphisms: the question of whether or not there exists an
$A$-trivialization of $X$ is the same for $\cC$ and $\cC^{inv}$.    
Next use Lemma~\ref{le:symm-extended} to replace $\cC^{inv}$ by an
extended symmetric monoidal category of type $(A,N)$, where
$A=\Pic(\cC)$ and $N=\Aut(S)$.  Finally, use
Proposition~\ref{pr:AN1}.
This shows the existence of an $A$-trivialization $\sigma$ of $\cC$ with
respect to $X$.

Suppose now that $\sigma'$ is another $A$-trivialization
of $\cC$ with respect to $X$.  Define 
\[ \theta_{a,b}=D\Bigl ( (\sigma_{a,b}')^{-1}\circ \sigma_{a,b} \Bigr )\in
\Aut(S)
\]
for each $a,b\in A$.  Note the resulting formula
$\sigma'_{a,b}=\sigma_{a,b}\ctens \theta_{a,b}$.  
Condition (1) in the definition of
$A$-trivialization shows that $\theta_{a,b}=\id_S$ if either $a$ or $b$
is zero.  Take the pentagonal diagram in condition (2) for $\sigma$
and let $C$ denote a composition going around the pentagon; let $C'$
denote the corresponding composition for $\sigma'$.  Note that
$C=C'=\id$ by commutativity of these diagrams.  But
if we replace each $\sigma'_{a,b}$ appearing in $C'$ 
with $\sigma_{a,b}\ctens \theta_{a,b}$,
then all of the $\theta$'s can be moved outside the composition by
Remark~\ref{re:move-around}.  This shows that $C'=C\ctens (\delta\theta)(a,b,c)$.  Since
$C=C'=\id$ we get that $\delta
\theta(a,b,c)=\id_S$ for every $a,b,c\in A$.  So $\theta\in
Z^2(A;\Aut(S))_{norm}$, and moreover it is easy to see that this gives a bijection
between $A$-trivializations and elements of $Z^2(A;\Aut(S))_{norm}$.  

Finally, we have seen how the trivializations $\sigma$ and $\sigma'$
each give rise to a ring structure on $\pi_*^A(S)$; write these as
$\pi_*^A(S)_\sigma$ and $\pi_*^A(S)_{\sigma'}$.  
Recall from the beginning of this
section that a standard isomorphism between these rings
depends on a fixed
map of sets $u\colon A\ra \Aut(S)$.  Let $F_u\colon \pi_*^A(S)\ra
\pi_*^A(S)$ be the map of $A$-graded abelian groups which sends
$f\colon X_a\ra S$ to $f\cdot u(a)$.  It is routine to check that
$F_u$ gives a ring isomorphism $\pi_*^A(S)_\sigma \ra
\pi_*^A(S)_{\sigma'}$ if and only if $\delta u=\theta$.  This
completes the proof.
\end{proof}


\appendix
\section{A short motivic application}

Here we give the proof of Proposition~\ref{pr:motivic-compare}.
We concentrate on the basic idea, ignoring technical details about the
foundations.

\medskip
Let $\Ho(\Spe)$ and $\Ho(\MotSp)$ denote the stable homotopy category
and the motivic stable homotopy category over $\C$, respectively.
These both have symmetric monoidal structures, with the units written
$S^0$ and $S^{0,0}$, respectively.  There is a
realization functor $\psi\colon \Ho(\MotSp)\ra \Ho(\Spe)$ that is
strong monoidal.  Let $X_1=S^{1,0}$ and $X_2=S^{1,1}$ be the standard
motivic spheres, and write $S^1$ for the classical suspension spectrum
of the circle.  Choose a model for $S^{-1}$ in $\Ho(\Spe)$ and an
isomorphism $\alpha\colon S^0 \ra S^{-1}\Smash S^1$.  Choose inverses
$X_1^*$ and $X_2^*$, and let us assume for simplicity that
$\psi(X_1^*)=S^{-1}$ and $\psi(X_2^*)=S^{-1}$ (equalities instead of
merely isomorphisms).  A little thought shows that one can choose isomorphisms
$\alpha_1\colon S^{0,0}\ra X_1^*\Smash X_1$ and 
$\alpha_2\colon S^{0,0}\ra X_2^*\Smash X_2$ that map
to $\alpha$ under $\psi$.  

Below we will write $Z=S^1$ to avoid having to write double exponents
like $(S^1)^a$.

\begin{proof}[Proof of Proposition~\ref{pr:motivic-compare}]
Let $f\colon X_1^a\Smash X_2^b\ra S$ and $g\colon X_1^r\Smash X_2^s\ra
S$.  Then $f\cdot g$ is the composite
\[ X_1^{a+r}\Smash X_2^{b+s}\llra{\phi} X_1^a\Smash X_2^b \Smash X_1^r
\Smash X_2^s \llra{f\Smash g} S\Smash S=S.
\]
The canonical isomorphism $\phi$ commutes the $X_2^b$ past the $X_1^r$
and then simplifies the resulting monomial by using associativity and
the $\alpha$ and $\hat{\alpha}$ maps (but without any more
commutations).
If we apply $\psi$ to this composite then we get the analagous
composite
\[ \xymatrixcolsep{2.9pc}\xymatrix{
Z^{a+r+b+s} \ar[r] & Z^a\Smash Z^b \Smash Z^r\Smash Z^s
\ar[r]^-{\psi(f)\Smash \psi(g)} & S\Smash S \ar@{=}[r] & S.
}
\]
Note that the first map in the composite is not a canonical map
anymore, and so we have dropped the label $\phi$.  Rather, this map
commutes the $Z^r$ past the $Z^b$.  If we were to compute
$\psi(f)\cdot \psi(g)$ in $\pi_*(S)$, however, we would get the
composite
\[ \xymatrixcolsep{2.9pc}\xymatrix{
Z^{a+r+b+s} \ar[r]^-\phi & Z^a\Smash Z^b \Smash Z^r\Smash Z^s
\ar[r]^-{\psi(f)\Smash \psi(g)} & S\Smash S \ar@{=}[r] & S.
}
\]
So we obtain a commutative diagram
\[ \xymatrixcolsep{2.9pc}\xymatrix{
Z^{a+r+b+s} \ar[r]^-{\psi(f\Smash g)}\ar[d]_{T} & S \\
Z^{a+r+b+s}\ar[ur]_{\psi(f)\Smash \psi(g)}
}
\]
where $T$ is the composite
\[ \xymatrixcolsep{2.4pc}\xymatrix{
Z^{a+r+b+s}\ar[r]^-\phi & Z^a\Smash Z^r\Smash Z^b\Smash Z^s
\ar[r]^{1\Smash t_{r,b}\Smash 1}  & Z^a\Smash
Z^b\Smash Z^r\Smash Z^s\ar[r]^-\phi & Z^{a+r+b+s}.
}
\]
Using Lemma~\ref{le:dumb} the triangle gives
$\psi(f\Smash g)=\psi(f)\Smash \psi(g) \Smash D(T)$.  We next compute
that
\[ D(T)=[T]_r=[\phi]_r\circ [1\Smash t_{r,b}\Smash 1]_r\circ
[\phi]_r=\id\circ [t_{r,b}]_r\circ
\id=\tau_{r,b}=\tau_1^{rb}=(-1)^{rb}
\]
(the first, second, and third equalities are by
Proposition~\ref{pr:self-map}, and the fifth is 
Proposition~\ref{pr:tau-ab}).
This yields the
desired result; one only needs to remember that the motivic bigrading
is set up so that $f\in \pi_{a+b,b}(S)$ and $g\in \pi_{r+s,s}(S)$, and
then one recovers the formula from the statement of the proposition.
\end{proof}

\begin{remark}
Note that in the setup of motivic homotopy groups we chose $X_1=S^{1,0}$ and
$X_2=S^{1,1}$.  The sign in Proposition~\ref{pr:motivic-compare}
actually depends on this choice.   We leave it as an exercise to check
that if we had chosen $X_1=S^{1,1}$ and
$X_2=S^{1,0}$ then the sign rule would be $\psi(fg)=(-1)^{(a-b)s}
\psi(f)\psi(g)$ for $f\in \pi_{a,b}(S)$ and $g\in\pi_{r,s}(S)$.  
This shows how sensitive sign formulas are to the choices in the 
bookkeeping.
\end{remark}


\bibliographystyle{amsalpha}

\end{document}